\documentclass[ reqno,]{amsart}
\usepackage{mathrsfs,amssymb, amscd,amsmath,amsthm}
\usepackage[enableskew,vcentermath]{youngtab}
\usepackage{multicol}\multicolsep=0pt

\newtheorem{Theorem}{Theorem}[section]
\newtheorem{lemma}[Theorem]{Lemma}
\newtheorem{cor}[Theorem]{Corollary}
\newtheorem{Prop}[Theorem]{Proposition}

 \theoremstyle{definition}
\newtheorem{example}[Theorem]{Example}

\newtheorem{Defn}[Theorem]{Definition}

\newtheorem{rem}[Theorem]{Remark}

\numberwithin{equation}{section}
\theoremstyle{definition}

\makeatletter
\def\enumerate{\begingroup\ifnum\@enumdepth>3\@toodeep\else
      \advance\@enumdepth\@ne
      \edef\@enumctr{enum\romannumeral\the\@enumdepth}%
      \topsep\z@\parskip\z@
      \list{\csname label\@enumctr\endcsname}
        {\@nmbrlisttrue\let\@listctr\@enumctr
         \parsep\z@\itemsep\z@\topsep\z@
         \setcounter{\@enumctr}{0}
         \def\makelabel##1{\hss\llap{\rm ##1}}
       }\fi}

\makeatother

\let\bar=\overline
\let\epsilon=\varepsilon
\def\({\big(}
\def\){\big)}

\def\C{\mathbb C}

\def\0{\underline{0}}

\def\H{\mathscr H}

\def\Sym{\mathfrak S}

\DeclareMathOperator{\End}{End}

\let\gdom\rhd
\let\gedom\unrhd

\def\Std{\mathscr{T}^{std}}

\def\s{\mathfrak s}
\def\ts{\tilde\s}
\def\t{\mathfrak t}
\def\u{\mathfrak u}
\def\v{\mathfrak v}

{\catcode`\|=\active
  \gdef\set#1{\mathinner{\lbrace\,{\mathcode`\|"8000%
                                   \let|\midvert #1}\,\rbrace}}
  \gdef\seT#1{\mathinner{\Big\lbrace\,{\mathcode`\|"8000%
                                   \let|\midverT #1}\,\Big\rbrace}}
}
\def\midvert{\egroup\mid\bgroup}
\def\midverT{\egroup\,\Big|\,\bgroup}

\def\Set[#1]#2|#3|{\Big\{\ #2\ \Big| \
           \vcenter{\hsize #1mm\centering #3}\Big\}}

 \providecommand{\og}{``}
\providecommand{\fg}{''} \providecommand{\smfandname}{and}

\newcommand{\bea}{\begin{eqnarray}}
\newcommand{\eea}{\end{eqnarray}}
\newcommand{\be}{\begin{eqnarray*}}
\newcommand{\ee}{\end{eqnarray*}}

\newcommand{\mfg}{{\mathfrak g}}
\newcommand{\fh}{{\mathfrak h}}

\newcommand{\Hom}{{\rm Hom}}

\newcommand{\U}{{\rm U}}

\let\gdom\rhd
\let\gedom\unrhd

\def\Std{\mathscr{T}^{std}}

\def\s{\mathfrak s}
\def\ts{\tilde\s}
\def\t{\mathfrak t}
\def\u{\mathfrak u}
\def\v{\mathfrak v}

\def\bft{\t}

\def\Hom{\text{Hom}}

\def\U{\mathbf U}

\def\d{\delta}

\def\l{\lambda}

\def\o{\omega}

\def\sc{\scriptstyle}
\def\ssc{\scriptscriptstyle}

\def\C{\mathbb{C}}

\def\Hom{{\rm Hom}}

\def\Set{{\rm Set}}

\def\OTIMES{{{\sc\!}\otimes{\sc\!}}}

\numberwithin{equation}{section}
\title[Isomorphisms between simple modules of degenerate cyclotomic Hecke algebras]{Isomorphisms between simple modules of degenerate  cyclotomic Hecke algebras}

\author{Hebing Rui {\normalfont \smfandname} Linliang \vspace*{-8pt}Song}

\address{H.~R.  College of Science, Harbin Institute of Technology, Shenzhen 508155, China} \email{hbrui@hit.edu.cn}
\address{L.~S.  College of Science, Harbin Institute of Technology, Shenzhen 508155, China} \email{songlinliang@hit.edu.cn
}
\thanks{Both of us are partially supported by NSFC}
\begin{document}
\begin{abstract}
We give explicit isomorphisms  between simple modules of degenerate cyclotomic Hecke algebras
defined via various cellular bases. A special case gives a generalized Mullineux involution in the degenerate case.
\end{abstract}

\sloppy \maketitle\vspace*{-30pt}

\section{Introduction}
 Throughout, we work over the complex field $\mathbb C$. For any  positive integer $r$, let   $\H_r^{\rm aff}$ be the degenerate affine Hecke algebra.
  By definition, $\H_r^{\rm aff}$ is the unital associative $\mathbb C$-algebra  generated by $s_1, \ldots, s_{r-1}$ and $x_1, \ldots, x_r$ which satisfy the following relations:
\begin{eqnarray} s_i^2=1, \text{ for $1\le i< r$,} &  x_i x_j=x_j x_i, \text{ if $1\le i,j\le r$,}  \\
s_is_j=s_j s_i, \text{ if $|i-j|>1$,}   & s_i x_j=x_j s_i, \text{ if $j\neq i, i+1$,}   \\
x_is_i-s_ix_{i+1}=-1,  \text{ if $1\le i<r$,}  &  s_i x_i-x_{i+1}s_i=-1, \text{ if $1\le i<r$,}  \\
s_is_{i+1}s_i=s_{i+1}s_is_{i+1}, \text{ if $1\le i< r-1$.} &
\end{eqnarray}
Let $\mathfrak S_r$ be the symmetric group on $r$ letters. Then the group algebra $\mathbb C\mathfrak S_r$ is isomorphic to the subalgebra of $\H_{ r}^{\rm aff}$ generated by
$\{s_i\mid  1\le i\le r-1\}$. Moreover, the isomorphism sends the simple transposition $(i, i+1)\in \mathfrak S_r$ to $s_i$ for all $1\le i\le r-1$.


 Let $\omega=(\o_1,  \ldots, \o_\ell)\in \mathbb C^\ell$. The \textit{degenerate cyclotomic Hecke algebra} or degenerate Hecke algebra of type $G(\ell, 1, r)$ is
\begin{equation}\label{dcha}   \H_{\ell,r}=\H_r^{\rm aff}/\langle f(x_1)\rangle,\end{equation}
 where $\langle f(x_1)\rangle$ is the two-sided ideal of $\H_r^{\rm aff}$ generated by $$f(x_1)=(x_1-\o_1)\cdots(x_1-\o_\ell).$$
 It is proved in  \cite{AMR} that $\H_{\ell, r}$ is a cellular algebra over the poset $\Lambda_{\ell, r}$ in the sense of \cite{GL}, where
 $\Lambda_{\ell, r}$ is the set of $\ell$-partitions of $r$. For each $\lambda\in \Lambda_{\ell, r}$, there is a cell module, say $S({\lambda})$,  on which  there is an invariant form $\phi_\lambda$. Let $\text{rad} \phi_\lambda$ be the radical of $\phi_\lambda$.
 By Graham-Lehrer's results on cellular algebras in \cite{GL}, the quotient module  $D(\lambda):=S({\lambda})/\text{rad}\phi_\lambda$ is either zero or simple and all non-zero  $D(\lambda)$'s give  a complete set of non-isomorphic simple $\H_{\ell, r}$-modules.

It is possible to show the existence of various cellular structures associated to $\mathscr H_{\ell, r}$.
 These structures, in turn, induce several ways to study the representation theory of the algebra. In particular, each of these structures depends on  one dimensional representations of certain  Young subgroups of  $\Sym_r$.
  For example,  a cellular basis  of $\H_{\ell, r}$ has been constructed   via trivial representations of certain Young subgroups of $\Sym_r$ in \cite{AMR}, whereas another one  is defined via   sign representations of certain Young subgroups of $\Sym_r$ in non-degenerate case in \cite[Remark~2.8]{DR}\footnote{See \cite{RSu2}  in degenerate case when $\ell=2$.}.
 Different structures really give different parameterizations of   the simple modules of the algebra.
It is quite  natural to ask how these parameterizations are related.


 Fix $\omega=(\omega_1, \ldots, \omega_{\ell})\in \mathbb C^\ell$ and a $01$-sequence   $\underline c=(c_1, \ldots, c_\ell)\in \{0,1\}^\ell$. We construct various cellular bases of  $\H_{\ell, r}$ in Corollary~\ref{Hc cellular}(a)-(b) via a class of  one dimensional  representations of certain Young subgroups of $\Sym_r$.
 For any $\xi\in\mathfrak S_\ell$, define $\omega^\xi=(\omega_{(1)\xi}, \ldots, \omega_{(\ell)\xi})$.
 If we use the $\omega^\xi$ instead of the $\omega$ in Corollary~\ref{Hc cellular}(a)-(b) for  the special case $\underline c=\underline {c_0}:=0^\ell$, we will  get another two kinds of cellular bases in Corollary~\ref{Hc cellular}(c)-(d).  The corresponding  simple $\mathscr H_{\ell, r}$-modules defined via the  cellular bases in
  Corollary~\ref{Hc cellular}(a)-(d) are denoted by  $ D^{\underline c}(\lambda)$,  $\tilde D^{\underline c}(\lambda)$,
$D^\xi({\lambda})$ and   $\tilde D^\xi({\lambda})$, respectively. The aim of this paper is to establish explicit isomorphisms between these simple modules.

   There are basically two cases to consider and each of them leads to a different proof.   We  establish the super Schur-Weyl duality between the general linear Lie superalgebras and the  degenerate cyclotomic Hecke algebras. Using some cellular bases of $\H_{\ell, r}$  in Corollary~\ref{Hc cellular}, we classify  highest weight vectors of certain tensor modules in super parabolic  category $\mathcal O$.
   Via Hom-functors, we establish    explicit relationships between the representation theory of general linear Lie (super)algebras and that of degenerate cyclotomic Hecke algebras.
    This is the second motivation of the paper.
More explicitly, such functors send
   parabolic (dual) Verma supermodules in (super) parabolic category $\mathcal O$ to  cell modules of degenerate cyclotomic  Hecke algebras.
       We remark that these results are available for  any $\underline c$.
 We  use Brundan-Losev-Webster's result on  the uniqueness of tensor product categorification to determine whether  $D^{\underline c}(\lambda) $ is isomorphic to  $D^{\underline {c_0}}(\mu)$ or not in Theorem~\ref{main1}. This is the first case of  our main result.   Secondly,  we establish an explicit isomorphism between simple modules $D^{\xi}(\lambda)$ and  $D^{1}(\mu)$ in Theorem~\ref{main2}. The idea of the proof of Theorem~\ref{main2} is similar to that of the first case.  The difference is that we use the isomorphism of
  crystals of some simple modules of some special  linear Lie algebra to determine whether  $D^\xi(\lambda)$ is isomorphic to $D^1(\mu)$ or not.
  This will give explicit isomorphisms between simple modules  $ D^{\underline c}(\alpha)$, $\tilde D^{\underline c}(\beta)$,
$D^\xi({\gamma})$ and   $\tilde D^\xi({\delta})$, $\alpha, \beta, \gamma, \delta\in \Lambda_{\ell, r}$  via routine arguments.
 A special case   gives a generalized Mullineux involution in the degenerate case.
  For the non-degenerate case, Jacon-Lecouvey~\cite{JL} determined whether $D^{\underline c}(\lambda)\cong \tilde D^{\underline c} (\mu)$ or not  when    $\underline c=0^\ell$.
  In this case, the  relationship between two labellings  is called a  \textit{generalized Mullineux involution}. See \cite{BO, Br0, BK1, FK, K0, M} for the Mullineux involutions on symmetric groups and  Hecke algebras.

By Brundan-Kleshchev's remarkable result in \cite{BK3}, the degenerate cyclotomic Hecke algebra is isomorphic to a generic version of the usual cyclotomic Hecke algebra. It is not clear to us whether  the results of this paper can be derived from this isomorphism directly and the study
of the associated problem in the non-degenerate case. One of the reasons is that Jacon and Lecouvey considered some special $\underline c$ and $\xi$.   However, it is possible to establish the corresponding results for usual cyclotomic Hecke algebras via certain results on quantum general linear (super)groups in generic cases.  In the root of unity cases, we do not know how to settle this problem yet.

We organize this paper as follows. In section~2, we recall some results on degenerate cyclotomic Hecke algebras.  In section~3, we classify  highest weight vectors of  certain tensor supermodules in super parabolic category $\mathcal O$ for general linear Lie superalgebras and   establish    explicit relationships between the representation theory of general linear Lie (super)algebras and that of degenerate cyclotomic Hecke algebras. In section~4, we recall  Brundan-Losev-Webster's results on tensor product categorifications  in~\cite{BLW}.
Via the results in section~3--4   together with Brundan-Kleshchev's results on  degenerate analogue of Ariki's categorification theorem   in \cite{BK}, we prove Theorem~\ref{main1} and Theorem~\ref{main2}, which  give explicit
isomorphisms between  simple  $\H_{\ell, r}$-modules defined via various cellular bases in Corollary~\ref{Hc cellular}.

\textbf{Acknowledgement:} Both of us wish to thank the referees for their detailed comments.

\section{Degenerate cyclotomic Hecke algebras }

\def\s{\mathfrak s}
\def\ts{\tilde\s}
\def\t{\mathfrak t}
\def\u{\mathfrak u}
\def\v{\mathfrak v}

Fix positive integers $\ell$ and $r$.  An \textit{$\ell$-composition} $\lambda$ of $r$ is of form $(\lambda^{(1)},  \ldots, \lambda^{(\ell)})$ where
$\lambda^{(i)}=(\lambda^{(i)}_1, \lambda^{(i)}_2, \ldots) $ is a sequence of non-negative integers and $|\lambda|=\sum_{i=1}^\ell |\lambda^{(i)}|=\sum_{i=1}^\ell \sum_j \lambda^{(i)}_j=r$.
If each $\lambda^{(i)}$  is    weakly decreasing, for any $1\le i\le \ell$, then $\lambda$ is called an\textit{ $\ell$-partition}. When $\ell=1$,  $\lambda$ is a usual partition of $r$.
The set  $\Lambda_{\ell, r}$  of $\ell$-partitions of $r$ is   a partially ordered set under the dominance order
$\gedom$, where $\lambda\gedom\mu$ if
$$\sum_{t=1}^{s-1}|\lambda^{(t)}|+\sum_{j=1}^k\lambda^{(s)}_k
     \ge\sum_{t=1}^{s-1}|\mu^{(t)}|+\sum_{j=1}^k\mu^{(s)}_k,$$
for all $1\le s\le \ell$ and all $k\ge0$. If $\lambda\gedom\mu$ and
$\lambda\ne\mu$ we  write $\lambda\gdom\mu$.

Suppose $\lambda\in \Lambda_{1, r}$. The {\it ~Young diagram} $Y(\lambda)$ is a
collection of boxes arranged in left-justified rows with $\lambda_i$
boxes in the $i$th row. If  $\lambda\in \Lambda_{\ell, r}$, then
 $Y(\lambda)=(Y(\lambda^{(1)}), \ldots, Y(\lambda^{(\ell)}))$.  A {\it $\lambda$-tableau} $\s=(\s_1,\ldots, \s_\ell)$ is
obtained by inserting the numbers $i,\, 1\le i\le r$ into $Y(\lambda)$ without
repetition and $\s$   is  {\it standard} if the
entries in  each $\s_i$  increase both from left to right in each row
and from top to bottom in each column.  Each standard $\lambda$-tableau $\s $ can be identified with an up $\lambda$-tableau $(\mathbf s_1,\ldots, \mathbf s_r)$  such that $\mathbf s_i$ is obtained from
$\s$ by removing the boxes  containing  the entries which are  strictly greater than $i$. In this case,  $\mathbf s_i$ is a standard $\mu$-tableau for some  $\ell$-partition $\mu$ of $i$. Abusing notation,
 we use   $\mathbf s_i$ to denote $\mu$.  So, $(\mathbf s_1,\ldots, \mathbf s_r)$ can be considered as  a sequence of $\ell$-partitions.
  Following \cite{Ma}, let $\unlhd $ be the dominance order  on $\Std(\lambda)$,   the
set of standard $\lambda$-tableaux,   such that
$\s\unlhd \t$ if $\mathbf s_i\unlhd \mathbf t_i$, for all $1\le i\le r$.
The maximal element in  $\Std(\lambda)$ is $\t^\lambda$ which is  obtained from  $Y(\lambda)$ by inserting  $1, 2, \ldots, r$ from left to right
along the rows of $Y(\lambda^{(1)})$ and then $Y(\lambda^{(2)})$ and so on. The minimal element in $\Std(\lambda)$ is $\t_\lambda$ which is  obtained from
$Y(\lambda)$ by inserting $1, 2, \ldots, r$ from top to bottom along the columns of $Y(\lambda^{(\ell)})$ and then
$Y(\lambda^{(\ell-1)})$, and so on. Since  $\mathfrak S_r$ acts on the right of  $\{1, 2, \ldots, r\}$, it induces an  $\mathfrak S_r$-action  on the right of a $\lambda$-tableau
$\s$. Write $w=d(\s)$ if $\t^\lambda w=\s $. For example,  if $\lambda=((3,2), (3,1))\in \Lambda_{2, 9}$ and $w=s_1s_2$, then
$$
\t^{\lambda}=\left(\ \young(123,45),\  \young(678,9) \ \right), \ \
 \t_{\lambda}=\left(\ \young(579,68), \ \young(134,2)\ \right), \ \   \t^{\lambda}w=\left( \ \young(312,45),\  \young(678,9) \ \right).$$
In this case, $\t^\lambda$ can be identified with the up $\lambda$-tableau $(\mathbf t_1, \mathbf t_2, \ldots, \mathbf t_9)$ such that $\mathbf t_i=((i), (0))$ if $i\in \{1,2, 3\}$ and
$\mathbf t_i=((3, i-3), (0))$ if $i\in \{4,5\}$, and $\mathbf t_i=((3,2), (i-5))$ if $i\in \{6, 7, 8\}$ and $\mathbf t_9=\lambda$.

 Suppose $\lambda=(\lambda_1, \lambda_2, \ldots, \lambda_k)\in \Lambda_{1, r}$. The dual of   $\lambda$ is $\lambda'$,  where $$\lambda'=(\mu_1, \mu_2, \ldots, \mu_{\lambda_1}) \text{ such that $\mu_j=\#\{k\mid \lambda_k\ge j\}$.}$$ If $\lambda\in  \Lambda_{\ell, r}$, then $\lambda'$, the dual of $\lambda$,  is     $(\mu^{(\ell)}, \mu^{(\ell-1)},\ldots, \mu^{(1)} ) $ where $\mu^{(i)}$ is the dual  of $\lambda^{(i)}$ for any $1\le i\le \ell$.  For each $\t\in \Std(\lambda)$, let  $\t'\in \Std(\lambda')$ be  such that the $s$th component of $\t'$ is obtained by interchanging  rows and columns of the $(\ell-s+1)$th component of $\t$, for all $1\leq s\leq \ell$.
By \cite[Lemma~5.1]{Ma1},  \begin{equation} \label{wlae} w_\lambda=d(\t)d(\t')^{-1}\text{ and } l(w_\lambda)=l(d(\t))+l(d(\t'))\end{equation}
 for all $\t\in \Std(\lambda)$, where $w_\lambda=d(\t_\lambda)$ and $l(w)$ is the length of $w$.
Later on, given $x, y\in \Sym_r$, we write $x\cdot y $ if $l(xy)=l(x)+l(y)$.   Let
$\Sym_{\lambda}$ be the row stabilizer of  $\t^\lambda$, which is called the   Young subgroup
of $\Sym_r$ with respect to  $\lambda$.
For example,  $\Sym_\lambda=\Sym_3\times \Sym_2\times \Sym_3\times \Sym_1$ if $\lambda=((3,2), (3,1))\in \Lambda_{2, 9}$.
 Let
\begin{equation}\label{xla} x_\lambda=x_{\lambda^{(1)}}\cdots x_{\lambda^{(\ell)}}\text{ and } y_\lambda=y_{\lambda^{(1)}}\cdots y_{\lambda^{(\ell)}},\end{equation}
 where  $x_{\lambda^{(i)}}=\sum_{w\in \Sym_{\lambda^{(i)}}}w$ and $y_{{\lambda^{(i)}}}=\sum_{w\in \Sym_{\lambda^{(i)}}} (-1)^{l(w)}w$
 and $\Sym_{\lambda^{(i)}}$ is the row stabilizer of $\t^\lambda_{{i}}$, the $i$th subtableau of $\t^\lambda$.
 Then $$x_\lambda w=x_\lambda \text{  and $y_\lambda w=(-1)^{l(w)} y_\lambda$ for all $w\in \mathfrak S_\lambda$.}$$
 So,   $\mathbb C x_\lambda$ (resp.,  $\mathbb C y_\lambda$)  is the  trivial (resp., sign) representation of $\Sym_\lambda$.

 Fix $(\omega_1, \ldots, \omega_\ell)\in \mathbb C^\ell$ and
 recall that $\mathscr H_{\ell, r}$ is the degenerate cyclotomic Hecke algebra in \eqref{dcha}.
Following \cite{GL, DJM, AMR},
define
\begin{equation} \label{pi}\pi_{[\lambda]}=\pi_{a_1,1}\pi_{a_2,2}\cdots \pi_{a_{\ell-1},{\ell-1}} \  \text{and} \ \
 \tilde \pi_{[\lambda]}=\pi_{a_1, \ell-2}\pi_{a_2,\ell-3}\cdots \pi_{a_{\ell-1},{0}},  \end{equation}
where $[\lambda]=[a_0, a_1, \ldots, a_\ell]$, $a_0=0$ and  $a_i=\sum_{j=1}^i |\lambda^{(j)}|$, for all  $1\le i\le \ell$ and
$$\pi_{a,i}=(x_1-\o_{i+1})(x_2-\o_{i+1})\cdots (x_a-\o_{i+1})$$ if $a>0$ and  $\pi_{0,i}=1$.
For any  $\s, \t\in \Std(\lambda)$, let
\begin{equation}\label{mst}
    m_{\s,\t}={d(\s)^{-1}} m_{\lambda}  {d(\t)} \text{ and }  n_{\s,\t}={d(\s)^{-1}}  n_{\lambda} {d(\t)},
\end{equation}
where $    m_{\lambda}= \pi_{[\lambda]} x_{\lambda}  $  and $  n_{\lambda}= \tilde \pi_{[\lambda]} y_{\lambda}$.

\begin{Theorem}\label{H cellular}\cite[Theorem~6.3]{AMR} Let  $\H_{\ell,r}$ be the degenerate cyclotomic Hecke algebra over $\mathbb C$ in \eqref{dcha}. \begin{enumerate} \item
$\set{m_{\s,\t}|\s,\t\in \Std(\lambda)\text{ and }\lambda\in \Lambda_{\ell, r}}$ is a
 cellular basis of~$\H_{\ell,r}$ in the sense of \cite[Definition~1.1]{GL},  \item  $\set{n_{\s,\t}|\s,\t\in \Std(\lambda)\text{ and }\lambda\in \Lambda_{\ell, r}}$ is a
  cellular basis of~$\H_{\ell,r}$,
  \end{enumerate} and the  required anti-involution  $*$ in \cite[Definition~1.1]{GL} is the anti-automorphism of  $\H_{\ell, r}$ which fixes  $s_1, \ldots, s_{r-1}$ and $x_1$.
\end{Theorem}
\begin{proof}We remark that (a) is  a special case of \cite[Theorem~6.3]{AMR} and (b) is the degenerate analog of \cite[Remark~2.8]{DR}. One can verify (b) similarly.\end{proof}

Fix a $01$-sequence   $\underline c=(c_1, \ldots, c_\ell)\in \{0,1\}^\ell$.  Let
  $x_\lambda^{\underline c}=x_{\lambda^{(1)}}^{c_1}\cdots x_{\lambda^{(\ell)}}^{c_\ell}$ and $y_\lambda^{\underline c}=y_{\lambda^{(1)}}^{c_1}\cdots y_{\lambda^{(\ell)}}^{c_\ell}$ where
  \begin{equation} \label{xci}  x^{c_i}_{\lambda^{(i)}}=\begin{cases} x_{\lambda^{(i)}}, &\text{if $c_i=0$,}\\
   y_{\lambda^{(i)}},  & \text{if $c_i=1$,}\\
   \end{cases}  \text{ and } y^{ c_i}_{\lambda^{(i)}}=\begin{cases} y_{\lambda^{(i)}}, &\text{if $c_i=0$,}\\
   x_{\lambda^{(i)}},  & \text{if $c_i=1$.}  \\
   \end{cases}  \end{equation}
 For any  $\s, \t\in \Std(\lambda)$,
 let  \begin{equation}\label{mstnst}
 m^{\underline c}_{\s,\t}=d(\s)^{-1} \pi_{[\lambda]} x_\lambda^{\underline c} d(\t), \text{ and $ n^{ \underline c}_{\s,\t}=d(\s)^{-1}\tilde  \pi_{[\lambda]} y_\lambda^{\underline c} d(\t)$.}\end{equation}
  Obviously, $m^{\underline c}_{\s,\t}=m_{\s,\t}$ and  $n^{\underline c}_{\s,\t}=n_{\s,\t}$ if $\underline c=0^\ell$.
For any $\xi\in\mathfrak S_\ell$, define  \begin{equation}\label{omxi} \omega^\xi=(\omega_{(1)\xi}, \ldots, \omega_{(\ell)\xi}).\end{equation}  Let $m^{\xi}_{\s,\t}$ (resp., $n^{\xi}_{\s,\t}$) be obtained from $m_{\s,\t}$ (resp., $n_{\s,\t}$)  by using the $\omega_{(i)\xi}$ instead of the $\omega_i$ in \eqref{pi}, for all $1\leq i\leq \ell$.

\begin{cor}\label{Hc cellular} Assume that    $\underline c\in \{0,1\}^\ell$ and $\xi\in \Sym_\ell$.
\begin{enumerate} \item
$\set{m^{\underline c}_{\s,\t}|\s,\t\in \Std(\lambda)\text{ and }\lambda\in \Lambda_{\ell, r}}$ is
a cellular basis of~$\H_{\ell,r}$, \item $\set{n^{\underline c}_{\s,\t}|\s,\t\in \Std(\lambda)\text{ and }\lambda\in \Lambda_{\ell, r}}$ is a
 cellular basis of~$\H_{\ell,r}$,\item  $\set{m^{\xi}_{\s,\t}|\s,\t\in \Std(\lambda)\text{ and }\lambda\in \Lambda_{\ell, r}}$
 is a cellular basis of $\H_{\ell,r}$, \item
  $\set{n^{\xi}_{\s,\t}|\s,\t\in \Std(\lambda)\text{ and }\lambda\in \Lambda_{\ell, r}}$ is a cellular basis of $\H_{\ell,r}$,\ \end{enumerate}  and the  required anti-involution  $*$ is the anti-automorphism given in Theorem~\ref{H cellular}.
\end{cor}
\begin{proof} Suppose  $\s \in \Std(\lambda)$. Write $[\lambda]=[a_0, a_1, \ldots, a_\ell]$.   There are  $\s_i\in \Std(\lambda^{(i)})$  for all $ 1\le i\le \ell$,  and $d$, a distinguished right coset representative of $\Sym_{[\lambda]}:=\mathfrak S_{a_1}\times \mathfrak S_{{a_2-a_1}}\times\ldots\times \mathfrak S_{a_\ell-a_{\ell-1}}$
 in $\Sym_r$,   such that
$$d(\s)=d(\s_1)d(\s_2)\cdots d(\s_\ell)d.$$
So, the transition matrix between the two sets in (a) and   Theorem~\ref{H cellular}(a)
 are determined by that between two cellular bases $\{(d(\s_1)\cdots d(\s_\ell))^{-1} x_{\lambda}^{\underline c}d(\t_1)\cdots d(\t_\ell)\mid \s_i,\t_i\in \Std(\lambda^{(i)})
 \text{ for all $1\le i\le \ell$} \} $ and $\{(d(\s_1)\cdots d(\s_\ell))^{-1} x_{\lambda}d(\t_1)\cdots d(\t_\ell)\mid \s_i,\t_i\in \Std(\lambda^{(i)})
 \text{ for all $1\le i\le \ell$} \} $ of  $\mathbb C \Sym_{[\lambda]}$.
 This shows that  $\set{m^{\underline c}_{\s,\t}|\s,\t\in \Std(\lambda)\text{ and }\lambda\in \Lambda_{\ell, r}}$ is a basis of $\H_{\ell, r}$. Mimicking the arguments of the proof of cellular basis for cyclotomic Hecke algebras  in \cite{DJM} yields (a). One can verify (b) similarly. Finally, (c)-(d) follow from Theorem~\ref{H cellular}.   \end{proof}

For each  $\lambda\in \Lambda_{\ell, r}$,  let $S^{\underline c}(\lambda)$
(resp., $\tilde {S}^{\underline c}(\lambda)$,  $S^\xi(\lambda)$,  $\tilde {S}^\xi(\lambda)$)
be the associated  right \textit{cell module} of
$\H_{\ell, r}$ with respect to the cellular basis of $\H_{\ell, r}$   in Corollary~\ref{Hc cellular}(a) (resp., (b)-(d)).
By \cite[Definition~2.1]{GL},  $S^{\underline c}(\lambda)$ has basis
$\set{m^{\underline c}_\s|\s\in\Std(\lambda)}$ and the action of
$\H_{\ell, r}$ on $S^{  \underline c}(\lambda)$ is given by
\begin{equation} \label{cellact} m^{\underline c}_\s a=\sum_{\u\in\Std(\lambda)}r_{\s,\u}( a^*)m^{\underline c}_\u, \quad \text{ for all } a\in \H_{\ell, r},  \end{equation}
where the  scalars $r_{\s,\u}(a^*)$'s $\in \mathbb C$ are as those in  \cite[Definition~1.1(C3)]{GL}.
By  \cite[Definition~2.3]{GL},   there is an invariant form $\phi^{\underline c}_\lambda:=\langle\ ,\ \rangle$ on $S^{\underline  c}(\lambda)$
such that
$$\langle m^{\underline c}_\s,m^{\underline c}_\t\rangle m^{\underline c}_{\u,\v}\equiv
        m^{\underline c}_{\u,\s}m^{\underline c}_{\t,\v}\pmod{\H_{\ell, r}^{\gdom\lambda}}, \text{ for all } \u, \v\in \Std(\lambda), $$
        where $\H_{\ell, r}^{\gdom\lambda}  =\mathbb C\text{-span}
      \set{m^{\underline c}_{\u,\v}|\mu\in \Lambda_{\ell, r}, \mu\gdom\lambda\text{ and }\u,\v\in \Std(\mu)}$.
 Let
$$D^{\underline c}({\lambda})=S^{\underline c}(\lambda)/\text{rad}\phi^{\underline c}_\lambda,$$
where $\text{rad}\phi^{\underline c}_\lambda$ is the radical of  $\phi_\lambda^{\underline c}$.
By \cite[Proposition~3.2, Theorem~3.4]{GL},  $D^{\underline c}({\lambda})$ is either zero or simple,  and moreover,
all non-zero $D^{\underline c}({\lambda})$'s give  a complete set of pair-wise non-isomorphic simple $\H_{\ell, r}$-modules.

 Similarly, we have $\tilde D^{\underline c}(\lambda)$ and
$D^\xi({\lambda})$ and   $\tilde D^\xi({\lambda})$ defined via  cellular bases in Corollary~\ref{Hc cellular}(b)-(d), respectively.
The main aim of this paper is  to establish  explicit  isomorphisms between these simple  $\H_{\ell, r}$-modules.

For any   right $\H_{\ell, r}$-module $M$, let  $M^\circledast$ be its   contragredient dual. It is  $\Hom_\mathbb C(M,\mathbb C)$ such that the right $\H_{\ell, r}$-action is given via
$$(fh)(m)=f(mh^*), \text{ $\text{for all }  f\in M^\circledast, h\in\H_{\ell,r}$ and $m\in M$,}$$ where $*$ is the anti-automorphism of
$\H_{\ell, r}$ in  Theorem~\ref{H cellular}.

Let $\mathbb C_\ell[x_1,\ldots,x_r]$ be the quotient of $\mathbb C[x_1,\ldots,x_r]$ by the relations $x_1^\ell=\ldots=x^\ell_r=0.$
The twisted tensor product algebra $\mathbb C_\ell[x_1,\ldots,x_r]\rtimes \mathfrak S_r$ is a graded algebra such that each $x_i$ is of degree 1 and each $w\in\mathfrak S_r$ is of degree 0.
By  \cite[Appendix~A]{BK2},    $\mathscr H_{\ell,r}$ is a filtered algebra with a filtration
$$\mathbb C\mathfrak S_r=F_0\mathscr H_{\ell,r}\subseteq F_1\mathscr H_{\ell,r}\subseteq\ldots \subseteq F_t\mathscr H_{\ell,r}=\mathscr H_{\ell,r}, $$
 such that $t=(\ell-1)r$, and   $F_k \mathscr H_{\ell,r}$ is the vector space spanned by   $$\{x_1^{i_1}\cdots x_r^{i_r}w\mid  0\le i_1,\ldots ,i_r\le \ell-1, i_1+\ldots+i_r\leq k,
 w\in\mathfrak S_r\}.$$
  Moreover, the associated graded algebra $\text{gr}\mathscr H_{\ell,r} $ is identified with $\mathbb C_\ell[x_1,\ldots,x_r]\rtimes \mathfrak S_r$.
Let  $$\text{gr}_t: F_t\mathscr H_{\ell,r}\rightarrow \mathbb C_\ell[x_1,\ldots,x_r]\rtimes \mathfrak S_r $$ be the map sending an element to its degree $t$ graded term.
Following  \cite[Appendix~A]{BK2},
define $$\tau:\mathbb C_\ell[x_1,\ldots,x_r]\rtimes \mathfrak S_r \rightarrow\mathbb C$$ to be the $\mathbb C$-linear map such that

\begin{equation}\label{defoftau}
\tau(x_1^{k_1}x_2^{k_2}\cdots x_r^{k_r} w)=\begin{cases} 1,\quad &\text{if }k_1=\ldots=k_r=\ell-1 \text{ and } w=1,\\
0, \quad &\text{otherwise},
\end{cases}
\end{equation}
where $w\in\Sym_r$ and $0\leq k_i\leq \ell-1$ for all $1\leq i\leq r$.

\begin{Theorem}\cite[Lemma~A.1,~Theorem~A.2]{BK2}\label{tauma} Let  $\hat\tau=\tau \circ \text{gr}_t$. The linear maps $ \tau$  and $\hat\tau$ are trace forms on $\text{gr}\mathscr H_{\ell,r} $ and $\mathscr H_{\ell,r} $, respectively. Moreover, $\hat\tau(h)=\hat\tau(h^*)$ for any $h\in \H_{\ell, r}$.
\end{Theorem}

For any $\lambda\in \Lambda_{\ell, r}$, let $w_{[\lambda]}\in \mathfrak S_r$  be defined by
\begin{equation}\label{wll}(a_{i-1}+l)w_{[\lambda]}=r-a_i+l, \text{ for all $i$ with $a_{i-1}<a_i$, $1\le l\le a_i-a_{i-1}$,}\end{equation}
where $[\lambda]=[a_0, a_1, \ldots, a_\ell]$ is given in \eqref{pi}. For example, if $\lambda=((3,2), (3,1))\in \Lambda_{2, 9}$, then $[\lambda]=[0, 5, 9]$ and
$$w_{[\lambda]}=\begin{pmatrix} 1 & 2 & 3 & 4 & 5 & 6 & 7 & 8 & 9\cr 5 & 6 & 7 & 8 & 9 & 1 & 2 & 3 & 4\cr
\end{pmatrix}.$$
 Let
$\bft^i$ (resp., $\bft_i$) be the $i$th subtableau of $\bft^\lambda$ (resp., $\bft^\lambda w_{[\lambda]}^{-1}$), and define $w_{(i)}$ by $\bft^i  w_{(i)} = \bft_i$. Likewise, if we define   $\tilde \bft^i$ (resp., $\tilde \bft_i$) to be the $i$th subtableau of $\bft^\lambda w_{[\lambda]}$ (resp., $\bft_\lambda$), and  $\tilde w_{(i)}$ with $\tilde \bft^i \tilde w_{(i)}=\tilde \bft_i$,  then
 \begin{equation}\label{wlaex} w_\lambda=w_{(1)} \cdot w_{(2)}\cdots w_{(\ell)}\cdot  w_{[\lambda]}=w_{[\lambda]}\cdot  \tilde w_{(\ell)}\cdot \tilde w_{(\ell-1)}\cdots \tilde w_{(1)}, \ \text{and $ w_{[\lambda]}^{-1} w_{(i)}w_{[\lambda]}=\tilde w_{\ell-i+1}.$} \end{equation}

\begin{lemma}\label{grzlambda}
$\emph{gr}_t(\pi_{[\lambda]} w_{[\lambda]}\tilde\pi_{[\lambda']})= x_1^{\ell-1}x_2^{\ell-1}\cdots x_r^{\ell-1}w_{[\lambda]}$ for any  $\lambda\in \Lambda_{\ell, r}$.
\end{lemma}
\begin{proof} Write $[\lambda]=[a_0,a_1,a_2,\ldots, a_\ell]$  and $[\lambda']=[b_0, b_1,b_2,\ldots, b_\ell]$. By  \eqref{pi},
the  term of
$\pi_{[\lambda]}$ (resp., $ \tilde\pi_{[\lambda']} $) with highest degree is $ x_{1,a_1}x_{1,a_2}\cdots x_{1,a_{\ell-1}}$ (resp.,  $ x_{1,b_1}x_{1,b_2}\cdots x_{1,b_{\ell-1}}$), where $x_{1, 0}=1$ and  $x_{1,b}:=x_1x_{2}\cdots x_b$ for $1\leq b\leq r$.
 Since $w^{-1} x_i w=x_{(i)w}$ in $\text{gr}\mathscr H_{\ell,r} $,   the result follows from \eqref{wll}.
\end{proof}

\begin{lemma}\label{tauzlambda}
For  each    $\underline c\in \{0,1\}^\ell$ and $\lambda\in \Lambda_{\ell, r}$, define $z_\lambda^{\underline  c}=m_\lambda^{\underline  c}w_\lambda n_{\lambda'}^{\underline c}$. Then
 $\hat\tau(  z_\lambda^{\underline c}w_\lambda^{-1})=1$.
\end{lemma}
\begin{proof} We have $$\begin{aligned}
\hat\tau(z_\lambda^{\underline c}w_\lambda^{-1})= &\hat\tau(x_\lambda^{\underline c}\pi_{[\lambda]}w_\lambda\tilde \pi_{[\lambda']}y_{\lambda'}^{\underline c} w_\lambda^{-1})\\
=&\tau \circ \text{gr}_t (\pi_{[\lambda]}w_\lambda\tilde \pi_{[\lambda']}y_{\lambda'}^{\underline c} w_\lambda^{-1}x_\lambda^{\underline c}) \ \  \text{( by Theorem~\ref{tauma}) }\\ =& \tau (x_1^{\ell-1}\cdots x_r^{\ell-1}w_\lambda y_{\lambda'}^{\underline c} w_\lambda^{-1}x_\lambda^{\underline c} )  \ \ \text{( by Lemma~\ref{grzlambda})} \\
=&1,
\end{aligned}
$$
where the last equality  follows from \eqref{defoftau} and  $
w_\lambda^{-1} \Sym_{\lambda} w_\lambda\cap \Sym_{\lambda'}=\{1\}$ in  \cite[Lemma~1.5]{DR}.\end{proof}

By Theorem~\ref{tauma},  there is  an associative symmetric bilinear form  $ \langle\ ,\ \rangle$ on $\H_{\ell,r}$ such that    $$\langle h_1,h_2\rangle=\hat\tau(h_1h_2^*), \text{ for all } h_1,h_2\in\H_{\ell,r}.$$
The following result was proved in \cite[Theorem~5.5]{Ma} for cyclotomic Hecke algebras. In the degenerate case, the result can be verified in a similar way.  The difference is that one needs to use $\hat \tau$ instead of the trace forms on cyclotomic Hecke algebras in \cite{Ma}.

\begin{lemma}\label{mnbilinear}  (cf. \cite[Theorem~5.5]{Ma})
Suppose that $\s,\t\in\Std(\lambda) $ and $\u,\v\in\Std(\mu)$. Then
\begin{equation}\label{kyyy}
\langle m^{\underline  c}_{\s,\t},  n^{\underline c}_{\u,\v}\rangle=\begin{cases}1,&\quad \text{if } (\u',\v' )=(\s,\t),\\
0, &\quad \text{if either $\u' \unrhd \s$ or $\v'\unrhd\t$ is not satisfied.}\end{cases}
\end{equation}
\end{lemma}
\begin{proof} By arguments similar to those  in  the proof of  \cite[lemma~5.4]{Ma},  $\v'  \unrhd \t$ if $ m^{\underline c}_{\s,\t}n^{\underline c}_{\v,\u}\neq0$. So, $\v' \unrhd \t$ if $ \langle m^{\underline  c}_{\s,\t},n^{\underline c}_{\u,\v}\rangle\neq 0$. By Theorem~\ref{tauma},  $\u'\unrhd \s$ if $ \langle m^{\underline  c}_{\s,\t},n^{\underline c}_{\u,\v}\rangle\neq 0$.
Suppose that $(\u' ,\v' )=(\s,\t)$. By \eqref{wlae} and Lemma~\ref{tauzlambda},
$\langle  m^{\underline  c}_{\s,\t}, n^{\underline  c}_{\s' ,\t' } \rangle= \hat\tau (m^{\underline  c}_{\s,\t}n^{\underline c}_{\t' ,\s'}) = \hat\tau ( w_\lambda^{-1} m^{\underline  c}_\lambda w_\lambda n_{\lambda'}^{\underline c})=1$.
\end{proof}

\begin{Prop}\label{dualofcells}(cf. \cite[Corollary~5.7]{Ma}) As  $\H_{\ell, r}$-modules, $S^{\underline c}(\lambda)^\circledast\cong \tilde {S}^{\underline c}(\lambda')$ for any $\lambda\in \Lambda_{\ell, r}$.
\end{Prop}

\begin{proof}

By Lemma~\ref{mnbilinear},  the invariant form $\langle\ ,\ \rangle$ on $\H_{\ell, r}$ induces an invariant  form $ \langle\ ,\ \rangle:S^{\underline c}(\lambda)\times \tilde {S}^{\underline c}(\lambda')\rightarrow \mathbb C$  such that
$$ \langle m_\s^{\underline c} + \mathscr H_{\ell,r}^{\rhd \lambda}, n_\t^{\underline c}+\tilde{\mathscr H}_{\ell,r}^{\rhd \lambda'} \rangle=\langle m_{\t^\lambda,\s}^{\underline c}  , n_{\t^{\lambda'}, \t}^{\underline c}\rangle.$$
By \eqref{kyyy},  the corresponding  Gram matrix is upper unitriangular, forcing
$S^{\underline c}(\lambda)^\circledast\cong \tilde {S}^{\underline c}(\lambda')$.
\end{proof}


  Mimicking the arguments of the proof of \cite[Theorem~2.9]{DR} yields the following result.

  \begin{lemma}\label{subcell} Suppose $\lambda\in \Lambda_{\ell, r}$.  Then
   $\tilde{S}^{\underline  c}(\lambda')\cong  S_{\lambda}^{\underline c}$, where $S_\lambda^{\underline c} =z^{\underline c}_{\lambda}\mathscr H_{\ell,r}$.
   Moreover, $S_\lambda^{\underline c}  $  has a basis $\{z^{\underline c}_{\lambda} d(\t)\mid \t\in \Std(\lambda')\}$.
   \end{lemma}

\section{Super Schur-Weyl duality }

In this section,  we establish    explicit relationships between the representation theory of general linear Lie (super)algebras and that of degenerate cyclotomic Hecke algebras.
Fix $\underline{n}=(n_1, \ldots, n_\ell)\in \mathbb Z_+^{\ell}$ and $\underline c=(c_1, \ldots, c_\ell)\in \{0,1\}^\ell$.
Let $U=\oplus_{i=1}^\ell U_i$, where  $U_i$  is a homogenous   complex superspace with  basis $\{u_{p_{i-1}+1}, \ldots, u_{p_{i}}\}$  such that

\begin{equation}\label{defofpi} p_0=0, \text{ and }  p_i=\sum_{j=1}^i n_j  \text{ for all $ 1\leq i\leq \ell$,}
 \end{equation}
 and the parity of each element in $U_i$ is   $\bar c_i\in \mathbb Z_2$, where $\mathbb Z_2=\mathbb Z/2\mathbb Z$.     Thus, $U$ is a superspace whose  even (resp., odd) dimension is  $n:=\sum_{ c_i=0} n_i$ (resp., $m:=\sum_{ c_i=1} n_i$).
  Let  $\mfg=\End_{\mathbb C} (U)$. The  matrix unit $e_{i, j}\in \mfg$ with respect to the basis $\{u_1,\ldots, u_{m+n}\}$ of $U$ is of
parity  $[e_{i, j}]=[i]+[j]$, where $[i]=\bar c_k\in \mathbb Z_2$ if $u_i\in U_k$ for some  $1\le k\le \ell$.
 The Lie bracket on $\mfg$ is
 \begin{equation}\label{lib}
 [e_{i, j},e_{k, l}]=\d_{j,k}e_{i, l}-(-1)^{([i]+[j])([k]+[l])}\delta_{l,i}e_{k,j},\end{equation}
 where $\delta_{j,k}=1$ if $j=k$ and $0$, otherwise.  Then $\mathfrak g$ is a Lie superalgebra which is isomorphic to the general linear Lie superalgebra $ \mathfrak{gl}_{n|m}$.
Let $\fh=\mathbb C\text{-}{\rm span}\{e_{i,i}\,\,|\, 1\le i\le m+n\}$ be a Cartan subalgebra of $\mfg$.
Its dual space $\fh^*$ has dual basis
 $\{\delta_i\,|\,1\le i\le  m+n\}$ such that $\delta_i (e_{j,j})=\delta_{i,j}$. Abusing  notation, we also denote by $\lambda$ an element in   $\fh^*$ and call it a  {\it weight}. So,  \begin{equation}\label{weight-}\l=\sum_{i=1}^{m+n} \l_i\delta_i,\ \  \l_i\in\C.\end{equation}
Let $\mathfrak b$ be the Borel subalgebra which consists of all upper triangular matrices.
 The root system of $\mfg$ is $$R=\{\delta_{i}-\delta_j\mid 1\le i\neq  j\le m+ n\}$$ and the set of   positive roots associated to $\mathfrak b$ is
  $$R^+=R^+_{\bar 0}\cup R_{\bar 1}^+=\{\delta_i-\delta_j\mid i<j\}.$$ The parity of $\delta_i-\delta_j$ is $[i]+[j]$.  There is a non-degenerate
bilinear form $(\ , \ )$ on $\mathfrak h^*$ such that $$(\delta_i, \delta_j)=(-1)^{[i]}\delta_{i, j}.$$ Let $\unrhd$ be the dominance order on $\mathfrak h^*$ such that $\lambda\unrhd \mu$ if $\lambda-\mu\in \mathbb N R^+$.
When we consider simple $\H_{\ell, r}$-modules,  we  assume  $\o\in \mathbb Z^\ell$ without loss of any generality. See \cite[Section~7.1]{K} for its explanation.
So, it suffices  to consider $\o\in \mathbb Z^\ell$ in the definition of $\H_{\ell, r}$. For this purpose, it suffices to consider
 $\lambda\in \mathfrak h^*_{\mathbb Z}$,  a linear combination of $\delta_i$'s with integral coefficients $\lambda_i$'s.
 For each $\lambda\in \mathfrak h^*_{\mathbb Z}$, following \cite{BLW}, define $$p_\lambda=\sum_{[i]=\bar 1} (\lambda, \delta_i)\in \mathbb Z_2.$$ The super category $\mathcal O$ is the category of all finitely generated $\mathfrak g$-supermodules $M=M_{\bar 0}\oplus M_{\bar 1}$ which are \begin{itemize} \item
locally finite dimensional over $\mathfrak b$, \item   $M=\oplus_{\lambda\in \mathfrak h_{\mathbb Z}^*} M_{\lambda, p_\lambda}$, where $M_{\lambda, p_\lambda}$ is the $\lambda$-weight space of $M_{p_\lambda} $ with respect to $\mathfrak h$.\end{itemize}
 In particular, the natural $\mathfrak g$-supermodule $U$ and its linear dual $U^*$ are two objects  in $\mathcal O$. For each $\lambda\in \mathfrak h_{\mathbb Z}^*$, let $$M(\lambda)=\U(\mfg)\otimes_{\U(\mathfrak b)} \mathbb C_{\lambda, p_\lambda}$$ be the Verma supermodule with  highest weight $\lambda$, where $\mathbb C_{\lambda, p_\lambda}$ is the $1$-dimensional $\mathfrak b$-supermodule of weight $\lambda$ with the parity   $p_\lambda$. Let $L(\lambda)$ be the simple head of $M(\lambda)$.

Let $\mathfrak p=\mathfrak l+\mathfrak b$, where  $\mathfrak l=\mathfrak {gl}_{n_1}\oplus  \mathfrak {gl}_{n_2}\oplus \ldots \oplus \mathfrak {gl}_{n_\ell}$. The super  parabolic category $\mathcal O^{\underline n, \underline c}$ is the full subcategory of $\mathcal O$ containing
all objects which are locally finite dimensional over $\mathfrak p$. Following \cite{BLW},   $\mathcal O^{\underline n, \underline c}$ is said to be of type $(\underline n, \underline c)$. Let
\begin{equation}\label{rho}\rho=-\sum_{1\leq i<j\leq m+n}(-1)^{[i]+[j]} \delta_j -\sum_{[i]=1} \delta_i.\end{equation}
The simple objects of $\mathcal O^{\underline n, \underline c}$ are indexed by the set of $\mathfrak p$-dominant  weights
 \begin{equation}\label{pdo}\Lambda_{\mathfrak p}=\{ \lambda\in \mathfrak h_{\mathbb Z}^*\mid (-1)^{c_k} (\lambda+\rho, \delta_j-\delta_{j+1})>0, 1\le k\le \ell \text{ and } p_{k-1}<j< p_k  \}. \end{equation}
 For each $\lambda\in \Lambda_{\mathfrak p}$, the corresponding  parabolic Verma supermodule is  $$M^{\mathfrak p}(\lambda)=\U(\mfg)\otimes_{\U(\mathfrak p)}  L_{\lambda},$$ where  $L_\lambda$ is the finite dimensional simple $\mathfrak l$-supermodule
of highest weight $\lambda$ with the parity $p_\lambda$.
Following \cite{BLW},  let $I_+=I\cup (I+1)$,  where $I\subseteq \mathbb Z$ is  a non-empty  interval of $\mathbb Z$. If $J$ is a finite sub-interval of  $ I$ such that $|J_+|\ge 2 \max\{n_i\mid 1\le i\le \ell\}$,
 let $\lambda_{J}\in \Lambda_{\mathfrak p}$ such that
\begin{equation}\label{kjw} \lambda_{J}+\rho=\sum_{i=1}^{m+n} \kappa_i\delta_i,\end{equation}
where
$$(\kappa_{p_i+1}, \kappa_{p_{i}+2}, \ldots, \kappa_{p_{i+1}})=\begin{cases} (n_i, n_i-1, \ldots, 1)+(\text{Inf}J-1)^{n_i},  & \text{if $c_i=0$,}\\
-(|J_+|-n_i+\text{Inf}J, |J_+|-n_i+\text{Inf}J+1, \ldots, |J_+|+\text{Inf}J -1), & \text{if $c_i=1$,}\\
\end{cases}
$$ and $\text{Inf}J$ is the minimal integer in $J$.
Then  $L(\lambda_{J})$ is the unique indecomposable object in the block containing it~\cite[Lemma~2.20]{BLW}. So, $L({\lambda_{J}})$, which is $M^{\mathfrak p}(\lambda_{J})$,  is both projective and injective and hence tilting in $\mathcal O^{\underline n, \underline c}$.
For any $r\in \mathbb N$, let
\begin{equation}\label{tens321}  T^{J,  r}:=L(\lambda_{J})\otimes U^{\otimes r}. \end{equation}
 Then $T^{J,  r}$ is a tilting module in $\mathcal O^{\underline n, \underline c}$.

Now, we want to establish an explicit relationship between cell modules of $\H_{\ell, r}$ and  certain parabolic Verma supermodules in some $\mathcal O^{\underline n, \underline c}$'s.
We need the explicit $\H_{\ell, r}$-action on  $T^{J, r}$ from \cite[Section~3]{BLW} as follows.
Recall  the definition of the  Casimir element  \begin{equation} \label{cas} C=\sum_{1\le i, j\le m+n} (-1)^{[j]}e_{i, j} e_{j,i}.  \end{equation}
It is a central element of $\U(\mfg)$.  Let $ \Omega=\frac12(\Delta(C)-C\otimes1-1\otimes C)$, where
$\Delta$ is the co-multiplication of $\U(\mathfrak g)$.  Then \begin{equation} \label{omega} \Omega=\sum_{1\le i,j\le m+n}{}(-1)^{[j]}e_{i, j}\OTIMES e_{j,i}.\end{equation} Order the positions of  the  tensor factors of  $T^{{J}, r}$ as $0, 1, 2, \ldots, r$ from left to right successively.
For $0\le a<b \le r$,  define
$\phi_{a, b}: \mfg^{\otimes2}\to \mfg^{\otimes(r+1)}$ by \begin{equation}\label{pi-ab}
\phi_{a,  b}(x\OTIMES y)=1\OTIMES\ldots\OTIMES1\OTIMES \stackrel{a\text{th}}{x}\OTIMES 1\OTIMES\ldots\OTIMES1\OTIMES \stackrel{b\text{th}}{y}\OTIMES1\OTIMES\ldots\OTIMES1\mbox{ for all  }x,y\in\mfg.\end{equation}
Then  $s_1, \ldots, s_{r-1}, x_1\in {\rm End}_{\mathcal O^{\underline n, \underline c}}( T^{J, r})$, where
\begin{equation}\label{operator--1}s_i=\phi_{i,i+1}(\Omega)|_{T^{J, r}}\ (1\le i<r),\quad
x_1=\phi_{0,1}(\Omega)|_{T^{J, r}}.
\end{equation}
 Later on, we always denote by  $u$  the  highest weight vector of $L(\lambda_J)$.

\begin{lemma}\label{polyofx}
\begin{enumerate}
\item $L(\lambda_J) \otimes U$ has a super parabolic Verma flag
\begin{equation}\label{filofvm}
0=M_0\subset M_1\subset M_2\subset \ldots\subset M_\ell=L(\lambda_J)\otimes U
\end{equation}
such that $M_i$ is generated by $\{u\otimes u_1, u\otimes u_{p_1+1},\ldots, u\otimes u_{p_{i-1}+1}\}$ and moreover,  $M_i/M_{i-1}\cong M^{\mathfrak p}(\lambda_{J}+\delta_{p_{i-1}+1})$ for all $1\le i\le \ell$, where $p_i$'s are given in  \eqref{defofpi}.
 \item  For all $1\le i\le \ell$, let  $\omega_i=(\lambda_{J}+\rho,\delta_{p_i+1})+(1-(-1)^{c_i})/2$. Then  $f_i(x_1):=\prod_{j=1}^i (x_1-\omega_j)$  acts on  $M_i$ as zero.

\end{enumerate}
\end{lemma}
\begin{proof} By  arguments similar to those in the proof of \cite[Theorem~3.6]{Hum},  $ L(\lambda_J) \otimes U$ has a  super parabolic Verma flag as required. Let $\delta:=\sum_{i=1}^{m+n}(-1)^{[i]}\delta_i$.
 It is well-known that the Casimir element $C$ acts  on   $M^{\mathfrak p}(\lambda)$   as the scalar $$c_\lambda=(\lambda,\lambda+2\rho+(n-m-1)\delta), $$
  and hence    $\Omega$  acts
on $M^{\mathfrak {p}}(\lambda_{J}+\delta_i)$ as  the scalar $$\omega_i=(c_{\lambda_{J}+\delta_{p_i+1}}-c_{\lambda_{J}}-n+m)/2.$$
So,  $\prod_{j=1}^i (x_1-\omega_j)$  acts on  $M_i$ as zero.
\end{proof}

Recall the definitions of $w_\lambda$ and $n_{\lambda'}^{\underline c}$ in \eqref{wlae} and \eqref{mstnst}, respectively.
\begin{Defn}\label{defofhigest}  Let $\Lambda_{\ell,r}^{\underline{n}}=\{\lambda\in \Lambda_{\ell, r}\mid l(\lambda^{(i)})\leq n_i,\text{ for all } 1\leq i\leq \ell\}$,  where $l(\lambda^{(i)})$, called   the length of $\lambda^{(i)}$,  is the number of non-zero parts of $\lambda^{(i)}$. For any   $\lambda\in \Lambda_{\ell,r}^{\underline{n}}$,  define \begin{enumerate} \item  $\tilde{\lambda}=\lambda_{J}+\sum_{k=1}^\ell\sum_{j=1}^{n_k} \lambda_j^{(k)} \delta_{p_{k-1}+j}$, where $\lambda_J$ is given in \eqref{kjw}.
\item  $\mathbf i_\lambda=(\mathbf i_{\lambda^{(1)}},\ldots, \mathbf i_{\lambda^{(\ell)}})$, where
 $\mathbf i_{\lambda^{(i)}}=((p_{i-1}+1)^{\lambda_1^{(i)}},(p_{i-1}+2)^{\lambda_2^{(i)}}, \ldots, (p_i)^{ \lambda_{n_i}^{(i)}})$,

 \item \label{defofvt} $
u_\t=u\otimes u_{{\lambda}}w_\lambda n_{\lambda'}^{\underline c}d(\t)\in T^{J,  r}$ for any   $\t\in\Std(\lambda')$,  where $u_\lambda=  u_{\mathbf i_{\lambda}}$ and  $u_{\mathbf i}= u_{i_1}\otimes u_{i_2}\otimes \ldots \otimes u_{i_r}$ if $\mathbf i\in I(m+n, r)=\{(i_1, i_2, \ldots, i_r)\mid 1\le i_j\le m+n, \text{ for all } 1\le j\le r\}$.
 \end{enumerate} \end{Defn}
 For example, if  $(\ell,\underline n, r)=(2,(4,5),8)$, then  $(p_1,p_2)=(4,9)$. Suppose  $\lambda=((2,1), (3,1,1))\in \Lambda_{2,8}^{\underline{n}}$. Then
\begin{itemize} \item $l(\lambda^{(1)})=2$,  $l(\lambda^{(2)})=3$ and  $\tilde \lambda=\lambda_J+2\delta_1+\delta_2+3\delta_5+\delta_6+\delta_7$,
 \item $\mathbf i_{\lambda^{(1)}}=(1,1,2)$, $\mathbf i_{\lambda^{(2)}}=(5,5,5,6,7)$ and $\mathbf i_\lambda=(1,1,2,5,5,5,6,7)$,
 \item  $u_\lambda=u_1\otimes u_1\otimes u_2\otimes u_5\otimes u_5\otimes u_5\otimes u_6\otimes u_7$. \end{itemize}
 By  Definition~\ref{defofhigest},   $u_\t\in T^{J,r}$,   a weight vector   with weight $\tilde{\lambda}$.
Note that $\lambda\in \Lambda_{1,r}^n$ in Definition~\ref{defofhigest}  can be identified with  the weight $\sum_{i=1}^n \lambda_i\delta_i$.

\begin{lemma}\label{liehwv}\begin{enumerate}\item Suppose $\lambda$ and $\mu$ are two compositions of $r$. Then  $x_\lambda \mathbb C \mathfrak S_r y_{\mu'}=0$ unless $\lambda\unlhd \mu$.
\item   There is a bijection between $\Lambda^n_{1,r}$ and the set of  dominant weights of $U^{\otimes r}$,  where $U$ is the natural $\mathfrak{gl}_{n\mid 0}$-module.
Further, the $\mathbb C$-space of $\mathfrak{gl}_{n|0}$-highest weight  vectors with weight $\lambda$ has  basis
$\{ u_{\lambda} w_\lambda y_{\lambda'} d(\t)\mid \t\in \Std(\lambda')\}$ .
\item  There is a bijection between $\Lambda^m_{1,r}$ and  the set of  dominant weights of $U^{\otimes r}$,   where $U$ is the natural $\mathfrak{gl}_{0\mid m}$-module.
Further, the $\mathbb C$-space of $\mathfrak{gl}_{0 \mid m}$-highest weight vectors with weight $\lambda$ has  basis
 $\{ u_{\lambda} w_\lambda x_{\lambda'} d(\t)\mid \t\in \Std(\lambda')\}$.\end{enumerate}
 \end{lemma}
\begin{proof} (a)-(b) are  well-known results.  See, e.g. \cite{RSu2}. (c) can be verified by the arguments similar to those in the proof of (b).\end{proof}
We  describe a basis of $T^{J, r}$ before we study highest weight vectors of it.
Let $$ B=  \dot \cup_{i=1}^\ell  \dot \cup_{j=1}^{i-1}  \mathbf p_i\times \mathbf p_j,$$ where  $\mathbf p_i=\{p_{i-1}+1, \ldots, p_i\}$ is a sub-interval of $\mathbb Z$.
Write \begin{equation}\label{eijv} e_{{\mathbf i, \mathbf j}}^{\mathbf l}:=e_{i_a,j_a}^{l_a}e_{i_{a-1},j_{a-1}}^{l_{a-1}}\ldots e_{i_1,j_1}^{l_1}, \end{equation}
where $(i_k,j_k)\in B$, $1\le k\le a$ and   $\mathbf l\in \mathbb N^a$.
In this case,  we identify $\mathbf i$ (resp., $\mathbf j$) with $(i_a, i_{a-1}, \ldots, i_1)$
(resp., $(j_a, j_{a-1}, \ldots, j_1)$) and  say that   both $\mathbf i$ and $\mathbf j$ are   of lengths $a$.  If $a=0$, we set $ e_{{\mathbf i, \mathbf j}}^{\mathbf l}=1$.
 Let
 $\preceq $ be the lexicographic order on  $B$ in the sense that  $(i_1,j_1)\preceq (i_2,j_2)$ if either $i_1\le i_2$ or $i_1=i_2$ and $ j_1\leq j_2$. If
 $(i_1,j_1)\preceq (i_2,j_2)$ and $(i_1,j_1)\neq (i_2,j_2)$, we write $(i_1,j_1)\prec (i_2,j_2)$.
The following result is trivial.

\begin{lemma}\label{basismct} Let $S$ be the set of all elements  $e_{\mathbf k\mathbf l}^{\mathbf m}u\otimes u_{\bf i}\in  T^{J,  r}$, where  \begin{itemize} \item   ${\bf i} \in  I(n+m,r):=\{(i_1,i_2,\ldots, i_r)\mid 1\leq i_k\leq m+n, 1\leq k\leq r\}$,
 \item $\mathbf k, \mathbf l$ are sequence of integers  with  lengthes  $a$ such that $(k_i, l_i)\in B$  and  $(k_{i},l_{i})\prec (k_{i+1} ,l_{i+1})$, for all $1\le i\le a-1$ and $a\in \mathbb N$,
 \item   $\mathbf m\in \mathbb N^a$  such that  $m_i\in\{0,1\}$ if $[k_i]+[l_i]=1$.\end{itemize}
Then  $S$ is a basis of $ T^{J,  r}$.\end{lemma}

Later on, we say that  $e_{\mathbf k\mathbf l}^{\mathbf m}u\otimes u_{\bf i}\in S$  is  of degree    $\sum_{i=1}^r {m_i}$ if $\mathbf m=(m_1,\ldots, m_r)$. Recall that $u$ is the  highest weight  vector of $L(\lambda_J)$.

\begin{lemma}\label{acofx} Suppose  $i\in \mathbf p_{k}$ and $h\ge 1$. Then
\begin{equation}
(u\otimes u_i)x_1^h=\begin{cases} 0, & \text{if $k\le h$, }\\ \epsilon
\sum_{s=1}^h\sum_{ j_s\in \mathbf  p_{i_s}} e_{j_{h-1},j_h}\cdots e_{j_1,j_2}e_{i,j_1} u\otimes u_{j_h}, &\text{otherwise,} \\ \end{cases}
\end{equation}
up to some terms, say $m\in L(\lambda_J)\otimes U$ with  degrees $\leq h-1$, where $1\leq j_h< \ldots<j_1< k $ and $\epsilon\in \{-1, 1\}$. If $h=1$, then  $m= (-1)^{[i]}\lambda_{J}(e_{i,i})u\otimes u_i$.
\end{lemma}

\begin{proof} Induction on $h$. \end{proof}

Recall that  $ u_\t$ in Definition~\ref{defofhigest}.
\begin{Theorem}\label{hiofcyche} If  $\lambda\in \Lambda_{\ell,r}$ such that  $r\le \min\{n_1, \ldots, n_\ell\}$, then   $\{u_\t\mid \t\in \Std(\lambda')\}$ is a basis of  $\mathbb C$-space $V_{\tilde \lambda}$ consisting  of all highest weight    vectors of $T^{J,  r}$ with weight $\tilde\lambda$.
\end{Theorem}

\begin{proof} 
In order to prove the result, it suffices to verify  \begin{enumerate}\item
 $e_{j, j+1} u_\t=0$ for all $1\le j\le m+n-1$,
  \item $\{u_\t\mid \t\in \Std(\lambda')\}$ is linearly  independent and  the cardinality of  $\{u_\t\mid \t\in \Std(\lambda')\}$ is equal to $\dim V_{\tilde \lambda}$.\end{enumerate}
When we  prove  (a), it  suffices to consider the special case  $\t=\t^\lambda$. If $e_{j,j+1}\in  \mathfrak l$, by
 Lemma~ \ref{liehwv}, we have $e_{j,j+1}u_{\t}=0$. If $e_{i, i+1}\not\in \mathfrak l$, then  $j=p_i$ for some $i$, $1\le i\le \ell-1$.
If  $u_{p_i+1}$ does not occur in $u_{\mathbf i_\lambda}$ as a tensor factor, where $\mathbf i_\lambda$ is given in Definition~\ref{defofhigest}(b),  we have   $e_{p_i,p_i+1}u_{\t}=0$.
Otherwise,  $\lambda^{(i+1)}\neq\emptyset$. Let $[\lambda]=[b_0, b_1, \ldots, b_\ell]$. Then
\begin{equation} \label{kkk1} \begin{aligned} e_{p_i,p_i+1}u_\t& =\sum_{ a=1}^{ \lambda^{(i+1)}_1} u\otimes u_{\mathbf i_{\lambda^{(1)}}}\otimes \ldots\otimes u_{\mathbf i_{\lambda^{(i)}}}\otimes u_{\mathbf i_a}\otimes u_{\mathbf i_{\lambda^{(i+2)}}}\otimes \ldots \otimes u_{\mathbf i_{\lambda^{(\ell)}}}w_\lambda n^{\underline c}_{\lambda'}\\
&=u\otimes u_{\mathbf i_{\lambda^{(1)}}}\otimes \ldots\otimes u_{\mathbf i_{\lambda^{(i)}}}\otimes u_{\mathbf i_1}\otimes u_{\mathbf i_{\lambda^{(i+2)}}}\otimes \ldots \otimes u_{\mathbf i_{\lambda^{(\ell)}}}hw_\lambda n^{\underline c}_{\lambda'}\\
&=u\otimes u_{\mathbf j} hw_\lambda n^{\underline c}_{\lambda'},\\
\end{aligned}\end{equation}
for some  $h$ in the group algebra of $ \mathfrak S_{|\lambda^{(1)}|}\times   \ldots\times \mathfrak S_{|\lambda^{(\ell)}|}$, where  $\mathbf i_a$ is obtained from $\mathbf i_{\lambda^{(i+1)}}$ by using   $p_i$  instead of  $p_i+1$ at $(b_i+a)$th position and $\mathbf j=(\mathbf i_{\lambda^{(1)}}, \ldots, \mathbf i_{\lambda^{(i)}}, \mathbf i_1, \mathbf i_{\lambda^{(i+2)}}, \ldots, \mathbf i_{\lambda^{(\ell)}})$.
By \eqref{wll},  $hw_{[\lambda]}=w_{[\lambda]}h_1$  and $ h_1\tilde{\pi}_{[\lambda']}=\tilde{\pi}_{[\lambda']}h_1$ for some $h_1$ in the group algebra of $\mathfrak S_{|\lambda^{(\ell)}|}\times\ldots\times
\mathfrak S_{|\lambda^{(1)}|}$.
Then   $$u\otimes u_{\mathbf j} w_{[\lambda]}\tilde{\pi}_{[\lambda']}
= u\otimes u_{\mathbf k}  \tilde{\pi}_{[\lambda']}
$$ where
$ u_{\mathbf k}= u_{\mathbf i_{\lambda^{(\ell)}}}\otimes \ldots \otimes u
_{\mathbf i_{\lambda^{(i+2)}}}\otimes u_{\mathbf i_1}\otimes u_{\mathbf i_{\lambda^{(i)}}}\otimes\ldots\otimes u_{\mathbf i_{\lambda^{(1)}}}
$
up to a sign.
 We have  $$u\otimes u_{\mathbf k} \tilde{\pi}_{[\lambda']}=u\otimes u_{\mathbf k} (1, a_{\ell-i-1}+1)^2 \tilde{\pi}_{[\lambda']},$$ where $[\lambda']=[a_0,a_1,a_2,\ldots,a_\ell]$.
Note that the  tensor factor of  $ u\otimes  u_{\mathbf k} (1, a_{\ell-i-1}+1)$ at  $1$-st position is $u_{p_i}$ and $$(1, a_{\ell-i-1}+1) \tilde{\pi}_{[\lambda']} =(x_1-\omega_1)\cdots(x_1-\omega_i) h\text{ for some $h\in \mathscr H_{\ell, r}$.}$$
  Since  $ u\otimes u_{p_i}\in M_i$, where  $M_i$ is  given  in \eqref{filofvm},  by  Lemma~\ref{polyofx}(b), $ u\otimes u_{p_i}(x_1-\omega_1)\cdots(x_1-\omega_i)
=0$.  So, $ u\otimes u_{\mathbf k} \tilde{\pi}_{[\lambda']}=0$, and hence $ e_{p_i,p_i+1}u_\t=0$.
This completes the proof of (a).

 We are going to prove that   $\{u_\t\mid \t\in \Std(\lambda')\}$ is linearly  independent.
Note that $ \tilde{\pi}_{[\lambda']}$ contains a unique term $x_1^{m_1}x_2^{m_2}\cdots x_r^{m_r}$ such that  $\sum_{i=1}^r m_i$ is maximal.
 So, it is enough to prove that $\{u\otimes u_{\lambda}w_\lambda x_1^{m_1}x_2^{m_2}\cdots x_r^{m_r} y^{\underline c}_{{\lambda'}} d(\t)\mid \t\in \Std(\lambda')\}$ is linearly  independent.

 Note that  each   $w\in\mathfrak S_{r}$ can be identified with a unique  permutation diagram as follows. If
 $$w=\begin{pmatrix} 1 & 2 & \ldots  & r\cr i_1 & i_2 &\ldots & i_r\cr
\end{pmatrix},$$
 then the permutation diagram associated to $w$ is the diagram with $2r$ vertices arranged in two rows such that there are $r$ vertices indexed by $1,2,\ldots,r$ from left to right at both top and bottom rows, respectively, and moreover, each vertex indexed by $k$ at the bottom row connects with exactly the vertex indexed by $i_k$ at the top row.
  For example, the permutation diagram associated to $s_i$ is
$$
%
%
\put(0,2){\tiny$\ssc\bullet$}
\put(0,45){\small$r$}
\put(0,-10){\small$ r$}
%
%
\put(0,38){\tiny$\ssc\bullet$}
\put(1,40){\line(0,-1){38}}
%
\put(-17,36){$\cdots$}
\put(-17,-2){$\cdots$}
\put(-55,0){
\put(-88,36){$\cdots\cdots$}
\put(-88,-2){$\cdots\cdots$}
\put(-50,0){\tiny$\ssc\bullet$}\put(-50,38){\tiny$\ssc\bullet$}
\put(-50,45){\small$\!\!\!\! i\!-\!1$}
\put(-50,-10){\small$\!\!\!\! i\!-\!1$}
\put(-48,40){\line(0,-1){38}}
\put(-20,0){\tiny$\ssc\bullet$}
\put(-20,38){\tiny$\ssc\bullet$}
\put(5,45){\small$\!\!\!\! i\!+\!1$}
\put(5,-10){\small$\!\!\!\! i\!+\!1$}
\put(-17,39){\line(1,-2){18}}
\put(-17,2){\line(1,2){18}}
\!\!\!\!\!
\put(10,0){\tiny$\ssc\bullet$}\put(10,38){\tiny$\ssc\bullet$}
\put(-15,45){\small$ i$}
\put(-15,-10){\small$ i$}
}
\put(-25,0){\tiny$\ssc\bullet$}\put(-25,38){\tiny$\ssc\bullet$}
\put(-25,45){\small$\!\!\!\! i\!+\!2$}
\put(-25,-10){\small$\!\!\!\! i\!+\!2$}
\put(-23,40){\line(0,-1){38}}
\put(-53,-22){Figure 1}
%
%
%
%
$$\\[-10pt]

 Denote by $[i,j]$  the edge of $w$ such that the top (resp., bottom) endpoint of $[i,j]$ is the $i$th (resp., $j$th) vertex. For example, the edges of $s_i$ in the above diagram are  $\{[i,i+1], [i+1,i], [k,k]\mid k\neq i,i+1 \}$.


 Write $ u\otimes u_{\mathbf l}x_1^{m_1}x_2^{m_2}\cdots x_r^{m_r}w$  as a linear combination of elements of $S$, where  $S$ is the
   basis of $T^{J,r}$  in Lemma~\ref{basismct}. Recall that  an   $e_{\mathbf k\mathbf l}^{\bf h}u\otimes u_{\bf i}\in S$  is  of degree    $\sum_{i=1}^r {h_i}$.
 In order to describe a term of $ u\otimes u_{\mathbf l}x_1^{m_1}x_2^{m_2}\cdots x_r^{m_r}w$  with the highest degree,   we define a labeled permutation diagram
 $ D_{w}$ with respect to $w$ as follows.

 First, we insert   $m_i$  beads at the $i$th  vertex at the top row of $D_{w}$, for all $1\le i\le r$.
Secondly, we label each vertex with some positive  integer as follows.
\begin{enumerate}
\item For any $i, 1\le i\le r$, the $i$th  vertex   at the top row of   $D_{w}$ is  labeled  with the positive integer $l_i$, where $\mathbf l:=\mathbf i_{\lambda}w_\lambda=(l_1, \ldots, l_r)$,
\item If $[i, j]$ is an edge of $D_{w}$ and $m_i=0$, then the labeling of $j$th vertex  is $l_i$.
 \item If $[i, j]$ is an edge of $D_{w}$ and  $m_i>0$, then $l_i\in \mathbf p_{m_i+1}$ and  the labeling of $j$th vertex  is $l_i-\sum_{m=1}^{m_i} n_{m_i+1-m}+\sum_{j=1}^{m_i}|\lambda^{(j)}|\in \mathbf p_1$ (see the definition of $\mathbf p_i$ above  \eqref{eijv}).  \end{enumerate}
    Since we are assuming $r\le \min\{n_1, n_2, \ldots, n_\ell\}$, the above is well-defined.  For  pairs  $(l_i, m_i)$, $1\leq i\leq r$ with $m_i>0$ (determined by the labeled permutation diagram as above), define
     \begin{equation} \label{operator0} \mathcal Y_{l_i, m_i, i}  =e_{l_i, l_i-n_{m_i}} e_{l_i-n_{m_i}, l_i-n_{m_i}-n_{m_i-1}} \cdots e_{l_i-\sum_{j=1}^{m_i-1}n_{m_i-j+1}, l_i-\sum_{j=1}^{m_i}n_{m_i-j+1}+\sum_{j=1}^{m_i-1}|\lambda^{(j)}|}.
  \end{equation}
If $m_i=0$, we define   $\mathcal Y_{l_i, m_i, i}:=1$.
Consider the ordered product
$\mathcal Y=\mathcal Y_{l_1, m_1, 1}\cdots \mathcal Y_{l_r, m_r, r}$.
Let $\mathbf i_w$ be the sequence of positive integers obtained by reading the labeling of vertices  at the bottom row
of the labeled  permutation diagram $ D_{w}$ from left to right.
By Lemma~\ref{acofx} and \cite[Remark~3.16]{RSong1}, the  $\mathcal Y u$ term of $ u\otimes u_{\mathbf l}x_1^{m_1}x_2^{m_2}\cdots x_r^{m_r} $ is $u_{\mathbf i_1}$,  where $\mathbf i_1$ is  $\mathbf i_w$ for $w=1$.
Note that  the coefficient of $u_{\mathbf i_1}d(\t)$ in $u_{\mathbf i_1}y^{\underline c}_{{\lambda'}}d(\s)$ is non-zero if and only if $\t=\s$.  So
the coefficient of $\mathcal Y u\otimes u_{\mathbf i_1}d(\t)$
in $u_{\s}$ is non-zero if and only if  $\t=\s$. This proves that
  $\{u_\t\mid \t\in \Std(\lambda')\}$ is linearly independent.

  Note that $\dim_{\mathbb C} V_{\tilde \lambda}=\dim_{\mathbb C} \Hom_{\mathcal O^{\underline n, \underline c}} (M^{\mathfrak p}(\tilde\lambda), T^{J,r})$.
   Since there is a bijection between $\Std(\lambda)$ and the set of up $\lambda$-tableaux, by standard arguments on a parabolic Verma flag of $ T^{J,r}$, we have  $ \dim_{\mathbb C} \Hom_{\mathcal O^{\underline n, \underline c}} (M^{\mathfrak p}(\tilde\lambda), T^{J,r})= \# \Std(\lambda')$,  proving   (b).
\end{proof}
\begin{example} Assume $(n_1, n_2,n_3, \ell, r )=(11, 12,13, 3, 11)$.  Suppose $\lambda=((2), (2,1),(4,2))$,
 $d(\t)=s_2 s_4s_6s_8s_9$ and $\mathbf l=(24,25,24,25,24,24,12,13,12,1,2)$. Then $\mathbf m=(2^6,1^3)$. The following labeled  diagram is $ D_{d(\t)}$. In this case,
\begin{enumerate}\item $\mathbf i_{d(\t)}=(6,6,7,6,7,3,6,3,1,4,2)$,
\item $\mathcal Y_{24, 2, 1}=\mathcal Y_{24, 2, 3}=\mathcal Y_{24, 2, 5}=\mathcal Y_{24, 2, 6}=e_{24,12}e_{12, 6}$, $\mathcal Y_{25, 2, 2}=\mathcal Y_{25, 2, 4}=e_{25,13}e_{13, 7}$,
$\mathcal Y_{12, 1, 7}=\mathcal Y_{12, 1, 9}=e_{12, 3}$, $\mathcal Y_{13, 1, 8}=e_{13, 4}$.\end{enumerate}
\end{example}

\unitlength 1mm 
\linethickness{0.4pt}
\ifx\plotpoint\undefined\newsavebox{\plotpoint}\fi 
\begin{picture}(132.275,52.875)(20,70)
\put(40,108){\line(0,-1){23.5}}
\multiput(46.75,108.75)(.0307370242,-.0865051903){275}{\line(0,-1){.0865051903}}
\multiput(54.75,108.75)(-.0336734694,-.0979591837){245}{\line(0,-1){.0979591837}}
\multiput(63.25,108.75)(.0307370242,-.0865051903){275}{\line(0,-1){.0865051903}}
\multiput(71.25,108.75)(-.0336734694,-.0979591837){245}{\line(0,-1){.0979591837}}
\multiput(88,108.75)(.0307370242,-.0765051903){315}{\line(0,-1){.0865051903}}
\multiput(97,108.75)(-.0336734694,-.0879591837){265}{\line(0,-1){.0979591837}}

\multiput(105,108.25)(.0390090199,-.0556906585){422}{\line(1,0){.0428790199}}
\multiput(115.25,108.25)(-.0330038262,-.0761275416){291}{\line(0,-1){.0411275416}}
\multiput(122.25,108.25)(-.0307078652,-.0821985019){274}{\line(0,-1){.0411985019}}
\put(131,108){\line(0,-1){23.5}}

\put(40,108){\circle*{0.8}}
\put(46.75,108.2){\circle*{0.8}}
\put(54.75,108.2){\circle*{0.8}}
\put(63.5,108.5){\circle*{0.8}}
\put(71.25,108.5){\circle*{0.8}}
\put(87.95,108.5){\circle*{0.8}}
\put(96.75,108.5){\circle*{0.8}}
\put(105.5,108.2){\circle*{0.8}}
\put(115.5,108.2){\circle*{0.8}}
\put(122.25,108.2){\circle*{0.8}}
\put(131,108.2){\circle*{0.8}}

\put(40,85){\circle*{0.8}}
\put(46.75,85){\circle*{0.8}}
\put(54.75,85){\circle*{0.8}}
\put(63.5,85){\circle*{0.8}}
\put(71.55,85){\circle*{0.8}}
\put(87.95,85){\circle*{0.8}}
\put(97.5,85){\circle*{0.8}}
\put(105.5,85){\circle*{0.8}}
\put(114.25,85){\circle*{0.8}}
\put(121.25,85){\circle*{0.8}}
\put(131,85){\circle*{0.8}}

\put(40,105){\circle*{2}}
\put(40,107){\circle*{2}}
\put(47.25,107){\circle*{2}}
\put(48.25,105){\circle*{2}}
\put(54.25,107){\circle*{2}}
\put(53.85,105){\circle*{2}}
\put(63.75,107){\circle*{2}}
\put(64.15,105){\circle*{2}}

\put(70.75,107){\circle*{2}}
\put(70.25,105){\circle*{2}}
\put(88.15,107){\circle*{2}}
\put(88.95,105){\circle*{2}}

\put(96.25,107){\circle*{2}}
\put(106.25,107){\circle*{2}}
\put(115.25,107){\circle*{2}}

\put(23.75,95){$D_{d(\t)}=$}
\put(38.75,110){24}
\put(46.05,110){25}
\put(54,110){24}
\put(61.75,110){25}
\put(69.5,110){24}
\put(86,110){24}
\put(95.25,110){12}
\put(103.5,110){13}
\put(114,110){12}
\put(120.75,110){1}
\put(129.75,110){2}

\put(38.25,81){6}
\put(46.05,81){6}
\put(54,81){7}
\put(62.75,81){6}
\put(70.5,81){7}
\put(86,81){3}
\put(95.25,81){6}
\put(103.5,81){3}
\put(113.25,81){1}
\put(120.75,81){4}
\put(129.5,81){2}
\put(133.5,81){.}

\put(80.5,75){Figure 2}

\end{picture}

\begin{Theorem}\label{kernel} Suppose $r\leq \min\{n_1, \ldots, n_\ell\}$ and  $\H_{\ell, r}=\H_r^{\rm aff}/\langle f(x_1)\rangle$, where $f(x_1)$ is $f_\ell(x_1)$   in Lemma~\ref{polyofx}(b). Then there is an algebra isomorphism
$\phi: \H_{\ell, r}\cong \End_{\mathcal O^{\underline n, \underline c}} (T^{J, r})^{op}$ sending the generators  $x_1$ and $s_1, \ldots, s_{r-1}$ of $\H_{\ell, r}$ to the operators
 $x_1$ and $s_1, \ldots, s_{r-1}$ in \eqref{operator--1}, respectively.
\end{Theorem}

\begin{proof} We claim that  $\phi(S)$  is linearly independent, where
\begin{equation}\label{bigs} S=\{ n^{\underline c}_{\s,\t}\mid \s,\t\in \Std(\lambda') \text{ and }  \lambda\in \Lambda_{\ell,r} \}.\end{equation}
Since   $ \sharp S={\rm dim}_{\mathbb C}\text{End}_{\mathcal O^{\underline n, \underline c}}(T^{J, r})$, we have $\dim_{\mathbb C} \phi(\H_{\ell, r})\ge {\rm dim}_{\mathbb C}\text{End}_{\mathcal O^{\underline n, \underline c}}(T^{J, .r})$, forcing
$\phi$ to be surjective and hence  $\phi$ is an isomorphism.

 Abusing  notation, we denote  $\phi(h)$ by $h$ for any $h\in \H_{\ell, r}$.
 By Theorem~\ref{hiofcyche}, $u_{\t^\lambda}$ is a highest weight vector, forcing $u\otimes u_{\lambda}w_{[\lambda]} \tilde\pi_{[\lambda']}\neq 0$. Suppose $w\in \Sym_r$.
 Then there are $w_1\in \Sym_{[\lambda]}, w_3\in \Sym_{[\lambda']}$ and $w_2\in
\mathfrak D_{[\lambda],[\lambda']}$, such that \begin{equation} \label{pres} w=w_1 w_2 w_3\end{equation}
where $\mathfrak D_{[\lambda],[\lambda']}$ is the distinguished
double coset representatives in $\mathfrak S_{[\lambda]}\backslash\mathfrak S_r/\mathfrak S_{[\lambda']}$.
 We have  $$u\otimes u_\lambda w_1=u\otimes u_{\mathbf j}$$ up to a sign  where $\mathbf j$ is obtained from $\mathbf i_\lambda$ by permuting some entries in each $\mathbf i_{\lambda^{(i)}}$.  Write $u\otimes u_{\mathbf j}w_2=u\otimes u_{\mathbf l}$ up to a sign  and $[\lambda']=[b_0, b_1, \ldots, b_\ell]$.

 If $l(w)\le l(w_{[\lambda]})$ and $w\neq  w_{[\lambda]}$, then
 $w_2\neq w_{[\lambda]}$. So
there is an $ i\in  \{b_{k-1}+1, \ldots, b_{k}\}$ such that
$ l_i\in \mathbf p_{\ell-j+1}$, for some  $j>k$.
In particular,  $(1, i) \in \Sym_{b_{k}}$. Write   $u\otimes u_{\mathbf l} (1, i) =u\otimes u_{\mathbf k}$ up to a sign.  Recall the definition $\pi_{a,b}$ in \eqref{pi}. Since $w_3\in \mathfrak S_{[\lambda']}$, $w_3 \tilde \pi_{[\lambda']}=\tilde \pi_{[\lambda']} w_3$ and hence
$$
u\otimes u_{\mathbf l} w_3 \tilde\pi_{[\lambda']}=u\otimes u_{\mathbf k}\pi_{b_{k}, \ell-k-1} \pi_{b_{k+1}, \ell-k-2}\cdots \pi_{b_{\ell-1}, 0} (1, i)   \pi_{b_{1}, \ell-2
}\cdots \pi_{b_{k-1}, \ell-k} w_3.
$$
 Note that $k_1=l_i$, where $\mathbf k=(k_1, \ldots, k_r)$. So,  $u\otimes u_{k_1}\in M_{\ell-j}\subset M_{\ell-k}$, where $M_{\ell-k}$ is given in \eqref{filofvm}. By  Lemma~\ref{polyofx}(b),  $u\otimes u_{\lambda}w \tilde\pi_{[\lambda']}=0$. So, we have
 \begin{equation}\label{killlemma12} u \otimes u_{\lambda}w \tilde\pi_{[\lambda']}=0, \text{  if $w\neq w_{[\lambda]}$ and $l(w)\le l(w_{[\lambda]})$.}\end{equation}

 By Theorem~\ref{hiofcyche}, $u_{\t}= u\otimes u_{\lambda} w_\lambda \tilde \pi_{[\lambda']} y_\lambda^{\underline c}$ is a highest weight vector. In particular, it is not $ 0$.
Suppose $w\in \mathfrak S_r$ such that $l(w)\le l(w_\lambda)$ and $w\neq w_\lambda$.
 Write $w=w_1w_2w_3$ as in \eqref{pres}. If $w_2\neq  w_{[\lambda]}$, by \eqref{killlemma12}, we have  \begin{equation} \label{kkk123}
 { u\otimes u_{\lambda} w \tilde \pi_{[\lambda']} y_\lambda^{\underline c} = 0}\end{equation}
 If $w_2=w_{[\lambda]}$, we  write  $w=w'w_{[\lambda]}$ such that   $w'\in \mathfrak S _{[\lambda]}=\mathfrak S_{|\lambda^{(1)}|}\times \ldots\times \mathfrak S_{|\lambda^{(\ell)}|}$. So,  $w'=h_1  \cdots h_\ell$, where   $h_i\in \mathfrak S_{|\lambda^{(i)}|}$, for all $1\le i\le \ell$.
  We have $$ u\otimes u_{\lambda} w  \tilde \pi_{[\lambda']} y^{\underline c}_{\lambda'} =u\otimes u_{\lambda} h_1h_2\cdots h_\ell y^{ c_1}_{\mu^{(1)}} \cdots   y^{ c_\ell}_{\mu^{(\ell)}} w_{[\lambda]}\tilde{\pi}_{[\lambda']}, $$
 where $\mu^{(i)}$ is the dual partition of  $\lambda^{(i)}$,  and   $x_{\mu^{(i)}}^{c_i}$ and $y_{\mu^{(i)}}^{c_i}$'s are defined in \eqref{xci}.
Note that $x^{c_i}_{\lambda^{(i)}}$ acts on $ u\otimes u_{\lambda}$ as a non-zero scalar.
By \cite[4.1]{DJ} (for $q=1$), $$ u\otimes u_{\lambda} h_1h_2\cdots h_\ell y^{ c_1}_{\mu^{(1)}} \cdots   y^{ c_\ell}_{\mu^{(\ell)}} w_{[\lambda]}\tilde{\pi}_{[\lambda']}= 0$$
 unless $h_i=w_{(i)}$ for all $1\leq i\leq \ell$, where $w_{(i)}$'s are given in \eqref{wlaex}. In the later case,  $w=w_\lambda$. So  \eqref{kkk123} is available
 for all $w\in \Sym_r$ such that $w\neq w_\lambda$ and $l(w)\le l(w_\lambda)$.


Suppose $\sum a_{\s,\t}  n^{\underline c}_{\s, \t}=0$ in  $\text{End}_{\mathcal O^{\underline n, \underline c}}(T^{J, r})^{op}$, where $a_{\s,\t}\in\mathbb C$ and the summation is over  $S$ in \eqref{bigs}.
 We claim  every $a_{\s,\t}$ is zero.
 Otherwise, define \begin{itemize}\item $S_1=\{\nu\in\Lambda_{\ell, r}\mid a_{\s,\t}\neq0 \text{ for some }\s,\t\in \Std(\nu')\}$,
   \item $S_2=\{\nu\in S_1\mid [\nu']=\mathbf a\}$, where  $\mathbf a=\min\{[\nu']\mid \nu\in S_1\}$ with respect to the lexicographic order.  \end{itemize}
 Let $ \lambda$ be a  maximal element in $S_2$ with respect to the dominant order $\unrhd$. Take a  $\s_0\in \Std(\lambda')$ such that $l(d(\s_0))$ is maximal among all $\s$ satisfying  $ a_{\s,\t}\neq0$.
  For any  $ \nu\in S_1\setminus S_2$, write $[\nu']=[b_0,b_1,b_2,\ldots,b_\ell]$ and $[\lambda']=[a_0,a_1,a_2,\ldots,a_\ell]$.
  Then  there is an $i\in \{1, 2, \ldots, \ell-1\}$ such that $a_j=b_j$ for $j<i$ and $a_i<b_i$.
  We can find  a $j_1\in \mathbf p_{j}$, $j\leq \ell-i$ such that $u_{j_1}$ is at $h$th position of  $u\otimes u_{\lambda} d(\s_0') d(\s)^{-1}$
  and $h\leq a_i+1$.  So
$$u\otimes u_{\lambda} d(\s_0' ) d(\s)^{-1}\tilde\pi_{[\nu']}=u\otimes u_{\bf j}\pi_{b_i, \ell-i-1}\cdots
\pi_{b_{\ell-1}, 0}(h,1) \pi_{b_1, \ell-2}\cdots \pi_{b_{i-1}, \ell-i }
$$
where $u_{\mathbf j}= u_{\mathbf i_\lambda}d(\s_0') d(\s)^{-1} (h,1)$ up to a sign.
 Since $(x_1-\omega_1)\cdots (x_1-\omega_{\ell-i})$ is a factor of $\pi_{b_i, \ell-i-1}\cdots
\pi_{b_{\ell-1}, 0}$,  by Lemma~\ref{polyofx}(b),
$u\otimes u_{\bf j} \pi_{b_i, \ell-i-1}\cdots \pi_{b_{\ell-1}, 0}   =0$. So,
\begin{equation}\label{key111} u\otimes u_{\lambda} d(\s_0') d(\s)^{-1}\tilde\pi_{[\nu']}=0, \end{equation}
if $[\lambda']$ is strictly less than $[\nu']$ with respect to the lexicographic order.
Acting $\sum a_{\s,\t}  n^{\underline c}_{\s,\t}$ on $u\otimes u_\lambda d(\s_0')$ and using \eqref{key111}  yields
 \begin{equation}\label{actionofnstoc} \sum_{\s,\t\in\Std(\nu'), \nu\in S_2} a_{\s,\t}   u\otimes u_\lambda d(\s_0') d(\s)^{-1}n^{\underline c}_{\nu'} d(\t)=0.\end{equation}
 Since $\nu\in S_2$, $[\nu']=[\lambda']$. By \eqref{killlemma12},   $u\otimes u_\lambda d(\s_0') d(\s)^{-1}n^{\underline c}_{\nu'}\neq0$ only if $d(\s_0') d(\s)^{-1}=w_1w_{[\lambda]} w_3$, where $(w_1, w_3)\in \Sym_{[\lambda]}\times \Sym_{[\lambda']}$.
 In this case, by Lemma~\ref{liehwv}(a), $u_\lambda w_1 w_{[\lambda]} w_3 y_{\nu'}^{\underline c}\neq 0$ only if $\lambda^{(i)}\unlhd \nu^{(i)}$ for all $1\leq i\leq \ell$.
 Since $ \lambda$ is    maximal in $S_2$, \eqref{actionofnstoc} can be written as
\begin{equation}\label{actionofnstoc1} \sum_{\t\in\Std(\lambda')} a_{\s,\t}   u\otimes u_\lambda d(\s_0') d(\s)^{-1}n^{\underline c}_{\lambda'} d(\t)=0.\end{equation}
Since $l(d(\s_0))$ is maximal,  by \eqref{wlae}, $l(d(\s_0')d(\s)^{-1})\le l(w_\lambda)$ and $d(\s_0')d(\s)^{-1}= w_\lambda$ if and only if $\s=\s_0$.  By \eqref{kkk123} and  \eqref{actionofnstoc1},
$$\sum_{\t\in \Std(\lambda')} a_{\s_0, \t} u_\t= \sum_{\t\in\Std(\lambda')} a_{\s_0,\t}   u\otimes u_\lambda w_\lambda n^{\underline c}_{\lambda'} d(\t)=0.$$
By Theorem~\ref{hiofcyche}, $u_\t$'s are linear independent. So $a_{\s_0, \t}=0$ for all $\t\in \Std(\lambda')$, a contradiction. \end{proof}

\begin{lemma}\label{moduleisomofhe}Suppose $r\leq \min\{n_1, \ldots, n_\ell\}$.
For any $\lambda\in\Lambda_{\ell,r}$, ${\rm Hom}_{\mathcal O^{\underline n, \underline c}} (M^{\mathfrak p}(\tilde\lambda),T^{J,  r})\cong \tilde {S}^{\underline c}(\lambda')$ as right $\mathscr H_{\ell,r}$-modules, where $\tilde \lambda$ is defined in Definition~\ref{defofhigest}.
\end{lemma}
\begin{proof}
For any $\t\in\Std(\lambda')$, by the universal property of parabolic Verma supermodules, we define   $f_{\t}\in {\rm Hom}_{\mathcal O^{\underline n, \underline c}}(M^{\mathfrak p}(\tilde\lambda),T^{J,  r})$ such that  $f_\t(m_{\tilde\lambda})=u_\t$, where $m_{\tilde \lambda}$ is the  highest weight vector of $M^{\mathfrak p}(\tilde \lambda)$ which is unique up to a non-zero multiple.
By Theorem~\ref{hiofcyche}, $ \{f_\t\mid  \t\in \Std(\lambda')\}$ is a basis of $\Hom_{\mathcal O^{\underline n, \underline c}}(M^{\mathfrak p}(\tilde\lambda),T^{J,  r})$.
Let $\phi: V_{\tilde\lambda}\rightarrow S^{\underline c}_\lambda$ be the linear isomorphism sending $u_\t$ to $ z^{\underline c}_\lambda d(\t)$, where $S^{\underline c}_\lambda=z^{\underline c}_\lambda\mathscr H_{\ell, r}$ (see \eqref{subcell}) and $z^{\underline c}_\lambda $ is defined in Lemma~\ref{tauzlambda}. We claim that   $ \phi$ is an  $\mathscr H_{\ell,r}$-homomorphism. Suppose $h\in\mathscr H_{\ell,r}$.  By Corollary~\ref{Hc cellular}, we have
\begin{equation}\label{yhaction}n^{\underline c}_{\lambda'}d(\t)h=\sum_{\s\in\mathscr{T}^{s}(\lambda') }a_\s n^{\underline c}_{\lambda'}d(\s)+\sum_{\s_1,\s_2\in\mathscr{T}^{s}(\nu),\nu\rhd \lambda'} a_{\s_1,\s_2} d(\s_2)^{-1}n^{\underline c}_\nu d(\s_1).
\end{equation}
Write $[a_0,a_1,\ldots,a_\ell]\preceq [b_0,b_1,\ldots,b_\ell]$ if $a_i\le b_i$ for all $1\le i\le \ell$. Given $\lambda, \mu\in \Lambda_{\ell, r}$, we have
\begin{equation}\label{pilabnu} \pi_{[\lambda]} \mathscr  H_{\ell, r} \tilde \pi_{[\nu]}=0 \text{ unless  $[\lambda]\preceq [\nu']$,}
\end{equation}
the  degenerate analog  of  \cite[Proposition~1.8]{DR}. One can prove it    as in
 \cite[Proposition~1.8]{DR}.
By Lemma~\ref{liehwv}(a) and (\ref{pilabnu}), we have   $m^{\underline c}_\lambda \mathscr  H_{\ell, r} n^{\underline c}_{\nu}=0$ if $\lambda, \nu\in \Lambda_{\ell, r}$ and  $\lambda\rhd \nu'$. Note that $\lambda\rhd \nu'$ if and only if $\lambda'\lhd \nu$.  So, $$\phi(u_\t)h= \sum_{\s\in\Std(\lambda') }a_\s m^{\underline c}_\lambda w_\lambda n^{\underline c}_{\lambda'}d(\s)=\sum_{\s\in\Std(\lambda') }a_\s\phi(u_\s).$$
In order to show that $\phi$ is an $\H_{\ell, r}$-homomorphism, it suffices to verify
\begin{equation}\label{equal0} u\otimes  u_{\mathbf i_\lambda} w_\lambda d(\s_2)^{-1}n^{\underline c}_\nu=0.\end{equation}
Since we are assuming $\nu\rhd \lambda'$, we have  either $ [\nu]=[\lambda']$ or $ [\lambda'] \prec[\nu]$. In the first case, \eqref{equal0} follows from Lemma~\ref{liehwv}(a).
 Write $[\nu]=[b_0,b_1,b_2,\ldots,b_\ell]$ and $[\lambda']=[a_0,a_1,a_2,\ldots,a_\ell]$.
In the second case, there is an $i$ such that $a_j=b_j$ for $j<i$ and $a_i<b_i$. By \eqref{key111},
$$u\otimes u_{\lambda} w_\lambda d(\s_2)^{-1}\tilde\pi_{[\nu]}=0.$$
So $u_{\t}h =\sum_{\s\in\Std(\lambda') }a_\s u_\s$ and $ \phi(u_\t)h=\phi(u_\t h)$. This shows that $\phi$ is an $\H_{\ell, r}$-homomorphism, and hence   $V_{\tilde \lambda}\cong S^{\underline c}_\lambda$ as right $\mathscr H_{\ell, r}$-modules. Via it, it is routine to check ${\rm Hom}_{\mathcal O^{\underline n, \underline c}} (M^{\mathfrak p}(\tilde \lambda),T^{J,  r})\cong S^{\underline c}_\lambda$. By Lemma~\ref{subcell}, the result follows.\end{proof}

\begin{rem} It has been proved in \cite{AST} that $\End_{\U}(M)$ is always a cellular algebra in the sense of \cite{GL} if $M$ is a tilting $\U$-module and $\U$ is an enveloping algebra of a Lie algebra or a  quantum group. The proofs of their results depend on the properties of tilting modules, standard modules and  costandard modules  in a highest weight category. So, their results
can be used in our case.
Moreover, by \cite[Definition~5.1]{AST},  $\Hom_\U(\Delta(\lambda), M)$ is the corresponding left cell module of
$\End_\U(M)$-module, where $\Delta(\lambda)$ is a standard module.
 Our result gives an explicit construction of $\Hom_\U(\Delta(\lambda), M)$ in the current case.
\end{rem}

  For any object $M$ in $\mathcal O^{\underline n, \underline c}$, {the homomorphism space}  $\text{Hom}_{\mathcal O^{\underline n, \underline c}}(T^{J, r},M)$ is naturally a left $\H_{\ell, r}$-module. Via the anti-involution $\ast$ on $\mathscr H_{\ell, r}$, it can be   considered as a right $\mathscr H_{\ell, r}$-module.
  The following result follows from  Lemma~\ref{moduleisomofhe} and \cite[Definition~5.1]{AST}. Of course, it can be  verified by arguments similar to those in the proof of  \cite[Lemma~5.11(a)]{RSong1}.

\begin{cor}\label{picell} Suppose  $r\leq \min\{n_1, \ldots, n_\ell\}$ and $\lambda\in\Lambda_{\ell,r}$. As right  $\H_{\ell, r}$-modules, $$\text{\rm Hom}_{\mathcal O^{\underline n, \underline c}}(T^{J, r},N^{\mathfrak p}(\tilde\lambda))\cong \tilde {S}^{\underline c}(\lambda')$$ where $N^{\mathfrak p}(\tilde\lambda)$ is the dual parabolic Verma supermodule with highest weight $\tilde\lambda$.
\end{cor}

\section{Tensor product categorification of $\mathfrak {sl}_J$-modules}
In the following, we  recall  some  results about tensor product categorification of $\mathfrak {sl}_J$-modules from \cite{BLW}. All the results can be found in \cite{BLW}  except Lemma~\ref{irrunderpi}.
Let $\mathfrak{sl}_I$ be the special linear Lie algebra consisting of complex trace zero matrices with rows and columns indexed by integers from $I_+$.  The exterior power $\Lambda^{n, c} V_I$ is the usual $\Lambda^n V_I$ (resp., $\Lambda^n W_I$ ) if $c=0$ (resp., $c=1$), where
 $V_I$ is the natural $\mathfrak{sl}_I$-module with  basis $\{ v_i\mid i\in I_+\}$ and   $W_I$ is its   linear dual  with  the dual basis $\{w_{i}| i\in I_+\}$.
Following \cite{BLW}, let $\Lambda_{I, n, c}$ be the set of $01$-sequences $\lambda=(\lambda_i)_{i\in I_+}$ such that $|\{i\in I_+\mid \lambda_i\neq c\}|=n$.

\begin{rem}  We use $\lambda$ to denote an $\ell$-partition, a weight and a $01$-sequence in this paper.  The meaning of  $\lambda$ can be read from the context.\end{rem}

 For each $\lambda\in \Lambda_{I, n, c}$, let
\begin{equation}\label{vlambdatensor}
v_{\lambda}=\begin{cases}v_{i_1}\wedge v_{i_2}\wedge \ldots \wedge v_{i_{n}}, \quad &\text{ if } c=0, \\
w_{i_1}\wedge w_{i_2}\wedge \ldots \wedge w_{i_{n}}, \quad &\text{ if } c=1. \end{cases}
\end{equation}
such that   $i_1<\ldots< i_{n}$ and $\lambda_{i_j}\neq c$,  $1\le j\le n$.
For example, if $I=\{1, 2, \ldots, N-1\}$ for $N\gg 0$ and $\lambda=(1,1,0,1,1,1, 0, \ldots, 0)\in \{0,1\}^N$, then $\lambda\in \Lambda_{I, 5, 0}$,   $\lambda\in \Lambda_{I, N-5, 1}$ and
\begin{equation} \label{01s} v_\lambda=\begin{cases} v_1\wedge v_2\wedge v_4\wedge v_5\wedge v_6, & \text{if $c=0$,}\\
 w_3\wedge w_7\wedge w_8\wedge\ldots \wedge  w_N, & \text{if $c=1$.}
 \end{cases} \end{equation}
It is easy to see that  $\Lambda^{ n,   c} V_I$ has  basis $\{v_\lambda\mid \lambda\in  \Lambda_{I,  n,  c} \}$.
For $(\underline n, \underline c)\in \mathbb Z_+^{\ell}\times \{0,1\}^{\ell}$, define

 $$\Lambda_{I, \underline n, \underline c}= \Lambda_{I, n_1, c_1}\times \ldots \times \Lambda_{I, n_\ell, c_\ell}.$$
If $\lambda=(\lambda^{(1)},\lambda^{(2)},\ldots, \lambda^{(\ell)} )\in \Lambda_{I, \underline n, \underline c}$,
define
\begin{equation}\label{defofvlambda1}
v_\lambda=v_{\lambda^{(1)}}\otimes v_{\lambda^{(2)}}\otimes \ldots \otimes v_{\lambda^{(\ell)}}.\end{equation}
Then  $\Lambda^{\underline n, \underline  c} V_I=\Lambda^{ n_1,  c_1} V_I \otimes \ldots\otimes \Lambda^{n_\ell, c_\ell} V_I$ has  basis $\{v_\lambda\mid \lambda\in  \Lambda_{I, \underline n, \underline c} \}$.
Obviously,  $|\lambda|=\sum_{i=1}^{\ell}|\lambda^{(i)}|$, where $|\lambda^{(i)}|$ (resp., $|\lambda|$) is the weight of $v_{ \lambda^{(i)}}$ (resp., $v_\lambda$).
By  \cite[Definition~2.9]{BLW},    there is a partial order $ \le$  on $\Lambda_{\mathbb Z, \underline n, \underline c}$ such that
$$\lambda \le \mu  \text{ if and only if } |\lambda|=|\mu| \text{ and } \sum_{i=1}^k|\lambda^{(i)}|\ge \sum_{i=1}^k|\mu^{(i)}| \text{ for all   } k.$$
Recall that $\le$ is the Bruhat order defined on $\Lambda_{\mathfrak p}$ in \cite[(3.9)]{BLW}. By \cite[(3.5)]{BLW},   there is a bijection, say $\phi_{\underline n, \underline c}$,  between  $(\Lambda_{\mathbb Z, \underline n, \underline c}, \le )$ and  $(\Lambda_{\mathfrak p}, \le) $ such that
\begin{equation}\label{inversephi} \phi^{-1}_{\underline n, \underline c}(\lambda)= (\lambda^{(1)},\lambda^{(2)},\ldots, \lambda^{(\ell)} ), \text{ for all }  \lambda\in \Lambda_{\mathfrak p}, \end{equation}  where  $\lambda^{(i)}=(\lambda_{ij})_{ j\in \mathbb Z}$ satisfies
\begin{equation}\label{philambda}\lambda_{ij}=\begin{cases} 1-c_i, &\text{if $j=(\lambda+\rho, \delta_l)$,  \text{for some} $ p_{i-1} < l\le p_i $,}\\  c_i, &\text{otherwise.}\\
\end{cases}
\end{equation}
By  \cite[Section~2.8]{BLW},   $\Lambda_J\cong\Lambda_{J, \underline n, \underline c}$,  where  $\Lambda_J=\{\lambda\in \Lambda_{I, \underline n, \underline c}\mid \lambda_{ij}=c_i, \text{whenever $j\not\in J_+$}\}$. The isomorphism sends $(\lambda_{ij})_{j\in I_+}$ to $(\lambda_{ij})_{ j\in J_+}$ for all $ 1\le i\le \ell$.
  Motivated by \cite{BLW}, let $\mathcal C_{J, \underline n, \underline c}$ be the Serre quotient category of $\mathcal O^{\underline n, \underline c}$ with respect to the set $\phi_{\underline n, \underline c}( \Lambda_{J, \underline n, \underline c})$. It follows from  \cite[Section~2]{BLW} that  $\mathcal C_{J, \underline n, \underline c}$ is a highest weight category with respect to the  weight poset $\phi_{\underline n, \underline c}( \Lambda_{J, \underline n, \underline c})$. Its simple    (resp.,  standard, costandard ) objects  are $\{L( \phi_{\underline n, \underline c}(\lambda))\mid \lambda\in \Lambda_{J, \underline n, \underline c}\}$  (resp., $\{M^{\mathfrak p}( \phi_{\underline n, \underline c}(\lambda))\mid \lambda\in \Lambda_{J, \underline n, \underline c}\}$, $\{N^{\mathfrak p}( \phi_{\underline n, \underline c}(\lambda))\mid \lambda\in \Lambda_{J, \underline n, \underline c}\}$), where $N^{\mathfrak p}( \phi_{\underline n, \underline c}(\lambda))$ is the dual parabolic Verma supermodule with the  highest weight $\phi_{\underline n, \underline c}(\lambda)$.
Let $[\mathcal C_{J, \underline n, \underline c}]:=\mathbb C\otimes_{\mathbb Z}K_0(\mathcal C_{J, \underline n, \underline c})$ where $K_0(\mathcal C_{J, \underline n, \underline c})$ is  the Grothendieck group of the category $\mathcal C_{J, \underline n, \underline c}$.

Recall that $V_J$ is the natural $\mathfrak{sl}_J$-module and $W_J$ is its linear dual.
There is an   $\mathfrak{sl}_{J}$-isomorphism
\begin{equation}\label{isomofslj}
\bigwedge^l W_J\cong \bigwedge^{|J_+|-l} V_J
\end{equation}
and the isomorphism sends  $w_{i_1}\wedge w_{i_2}\wedge \ldots \wedge w_{i_l}$  to $v_{j_1}\wedge v_{j_2}\wedge \ldots \wedge v_{j_{|J_+|-l}}$, where $i_1<i_2<\ldots <i_l$, $j_1<j_2<\ldots < j_{|J_+|-l}$ and
\begin{equation}\label{conditionofj}
  \{j_1,j_2,\ldots,j_{|J_+|-l} \}= J_+\setminus \{i_1,i_2,\ldots, i_l\}.
\end{equation}
By  \eqref{vlambdatensor}, this isomorphism   sends $v_\lambda\in \bigwedge^l W_J$ to $v_\lambda\in  \bigwedge^{|J_+|-l} V_J$, where
 $\lambda\in \Lambda_{J,l,1}$. Note that $\Lambda_{J,l,1}$ can be identified with $\Lambda_{J,|J_+|-l,0}$.
 One can understand it via \eqref{01s}. So,  $$\Lambda^{\underline n, \underline c}V_{J}\cong \Lambda^{\tilde{\underline n}, \underline{c_0}}V_{J}$$ as $\mathfrak{sl}_{J}$-modules where \begin{equation}\label{definitionoftildeni}
 \tilde n_i=\left\{
              \begin{array}{ll}
                n_i, & \hbox{ if $c_i=0$,} \\
                |J_+|-n_i, & \hbox{ if $c_i=1$.}
              \end{array}
            \right.
 \end{equation}
 Moreover, this  isomorphism sends $v_\lambda$ to $v_\lambda$. In the following, $(\bigwedge^{\underline n, \underline c}V_{J}, \Lambda_{J, \underline n, \underline c})$ will be identified  with $(\bigwedge^{\tilde{\underline n}, \underline{c_0}}V_{J}, \Lambda_{J, \tilde{\underline n}, \underline{c_0}})$.

\begin{Theorem}\label{unicate}\cite[Theorem~2.12, Theorem~3.10, Lemma~2.19, Corollary~5.30]{BLW}
 \begin{enumerate}

 \item There is an $\mathfrak{sl}_{J}$-isomorphism
$\bigwedge^{\underline n, \underline c}V_{J}\cong [\mathcal C_{J, \underline n, \underline c}],$
sending  $v_\lambda$  (resp., $b^*_\lambda$)  to $[M^{\mathfrak p}( \phi_{\underline n, \underline c}(\lambda))]$ (resp.,
 $[L(\phi_{\underline n, \underline c}(\lambda))]$) for any $\lambda\in \Lambda_{J, \underline n, \underline c}$, where $\{b^*_\lambda \mid \lambda \in \Lambda_{J, \underline n, \underline c}\}$
is the dual canonical basis of  $\Lambda^{\underline n, \underline c}V_{J}$.

\item There is an $\mathfrak{sl}_{J}$-isomorphism $[\mathcal C_{J, \underline n, \underline c}]\cong [\mathcal C_{J, \tilde{\underline n}, \underline{c_0}}]$ sending
        $[M^{\mathfrak p}( \phi_{\underline n, \underline c}(\lambda))]$ (resp., $[N^{\mathfrak p}( \phi_{\underline n, \underline c}(\lambda))]$, $[L(\phi_{\underline n, \underline c}(\lambda))]$) to $[M^{\mathfrak p}( \phi_{\tilde{\underline n}, \underline{c_0}}(\lambda))]$  ( resp., $[N^{\mathfrak p}( \phi_{\tilde{\underline n}, \underline{c_0}}(\lambda))], [L(\phi_{\tilde{\underline n}, \underline{c_0}}(\lambda))]$) for any $\lambda\in \Lambda_{J, \underline n, \underline c}$.
\end{enumerate}
\end{Theorem}

We need some definitions to state the results for the blocks of $\H_{\ell, r}$ (see~\cite[Section~3.2]{BK}).
Any finite dimensional $\H_{\ell, r}$-module $M$ has a generalized weight spaces decomposition
$$M=\bigoplus_{\mathbf i\in \mathbb Z^r}M_{\mathbf i} $$
where $M_{\mathbf i}=\{v\in M\mid (x_k-i_k)^Nv=0 \text{ for all } k=1,2,\ldots, r
\text{ and } N\gg 0\}$. Assuming  $M$ to be  the left regular   $\H_{\ell, r}$-module  yields  a system
of mutually orthogonal idempotents $\{e(\mathbf i)\mid \mathbf i\in \mathbb Z^r\} $ of  $\H_{\ell, r}$ such that $ e(\mathbf i)N=N_{\mathbf i}$ for any finite dimensional $\mathscr H_{\ell, r}$-module $N$.

Let $\{\alpha_i\mid i\in J\}$ be the simple root system  of $\mathfrak {sl}_J$. Then $\alpha=\delta_i-\delta_{i+1}$.
 Set $Q_+:=\sum_{i\in J}\mathbb Z_{\geq 0}\alpha_i$. For any $\alpha\in Q_+$ with height $\text{ht}(\alpha)=r$, write $J^\alpha=\{\mathbf i\in J^r\mid \alpha_{i_1}+\alpha_{i_2}+\ldots+ \alpha_{i_r}=\alpha\}$. Let $e_\alpha= \sum_{ \mathbf i\in \mathbf J^\alpha  }e(\mathbf i)$.    By \cite[Section~3.2]{BK}, $e_\alpha$ is either $0$ or the primitive central idempotent of $\H_{\ell, r}$  with respect to $\alpha$.
 Define   \begin{equation}\label{block1}
  \mathscr H_{\ell, \alpha}:= e_\alpha\mathscr H_{\ell,r}.
  \end{equation}
 Then $\H_{\ell, \alpha} $ is either $0$ or a single  block of $\mathscr H_{\ell,r}$ corresponding to $\alpha$.
    By \cite[Theorem~2.14]{BLW}
 \begin{equation} \label{block}\End_{\mathcal C_{J, \underline n, \underline c}} (T^{J, r})^{op}\cong \bigoplus_{\alpha\in Q_+, \text{ht}(\alpha)=r}\mathscr H_{\ell, \alpha},\end{equation}
where $T^{J, r}$ is given in \eqref{tens321}.
For any $\lambda\in \Lambda_{J, \underline n, \underline c}$, by  a special case of \cite[(2.15)]{BLW},
  \begin{equation}\label{isomoandc}
  \begin{aligned}
    & {\rm Hom}_{\mathcal C_{J, \underline n, \underline c}} (M^{\mathfrak p}(\phi_{\underline n, \underline c}(\lambda)),T^{J,  r})\cong  {\rm Hom}_{\mathcal O^{\underline n, \underline c}} (M^{\mathfrak p}(\phi_{\underline n, \underline c}(\lambda)),T^{J,  r}) , \\
   & {\rm Hom}_{\mathcal C_{J, \underline n, \underline c}} (T^{J,  r},N^{\mathfrak p}(\phi_{\underline n, \underline c}(\lambda)))\cong  {\rm Hom}_{\mathcal O^{\underline n, \underline c}} (T^{J,  r}, N^{\mathfrak p}(\phi_{\underline n, \underline c}(\lambda))),  \\
  & {\rm Hom}_{\mathcal C_{J, \underline n, \underline c}} (L(\phi_{\underline n, \underline c}(\mu)),N^{\mathfrak p}(\phi_{\underline n, \underline c}(\lambda)) )\cong  {\rm Hom}_{\mathcal O^{\underline n, \underline c}}  (L(\phi_{\underline n, \underline c}(\mu)),N^{\mathfrak p}(\phi_{\underline n, \underline c}(\lambda)) ).
\end{aligned}
\end{equation}

Later on,    we do not assume $r\leq \min\{n_1, \ldots, n_\ell\}$.
 It is known   that
there is a crystal  structure on $\Lambda_{J, \underline n, \underline c}$. Let $\Lambda^o$ be the vertex set of the connected component of the crystal graph containing $\kappa_J$,  where  \begin{equation}\label{kaj} \kappa_J=\phi^{-1}_{\underline n, \underline c}(\lambda_J)\end{equation}  and $\lambda_J\in \Lambda_{\mathfrak p}$  in \eqref{kjw}   (see \cite[section~2.10]{BLW}). Let $\pi_{\underline n, \underline c }$ be the   Schur functor $\Hom_{\mathcal C_{J, \underline n, \underline c}}(T^{J, r},-)$.

\begin{lemma}\label{irrunderpi}
 Suppose $\lambda\in\Lambda_{J, \underline n, \underline c}$. Then  $\pi_{\underline n, \underline c}(L(\phi_{\underline n, \underline c}(\lambda))=0$ unless $\lambda\in \Lambda^o$.
  In the later case, if $\text{wt}(v_{\lambda})=\alpha$ with $\text{ht}(\alpha)=r$ and $r\leq\min\{n_1, \ldots, n_\ell\}$,
then $\pi_{\underline n, \underline c}(L(\phi_{\underline n, \underline c}(\lambda))\cong D^{\underline c}(\lambda)$.
\end{lemma}
\begin{proof} Since $T^{J,  r}$ is tilting, $L(\phi_{\underline n, \underline c}(\lambda))$ is a constituent of  socle of  $T^{J,  r}$ if and  only if
it is a constituent of  head of $T^{J,  r}$. By \cite[Theorem~2.24]{BLW},  $\lambda\in \Lambda^o$ and hence $\pi_{\underline n, \underline c}(L(\phi_{\underline n, \underline c}(\lambda)))\neq0$ if and only if $\lambda\in \Lambda^o$.
Since  $T^{J, r}$ is   projective in $C_{J, \underline n, \underline c}$,
 all non-zero   $\pi_{\underline n, \underline c}(L(\phi_{\underline n, \underline c}(\lambda)))$'s give  a complete set of  non-isomorphic irreducible $\End_{C_{J, \underline n, \underline c}}(T^{J,  r})^{op}$--modules.
 Using \cite[Corollary~3.1c]{BDK} and ~\eqref{block} yields
\begin{equation}\label{homone}
\Hom_{\mathcal C_{J,\underline n, \underline c}}(L(\phi_{\underline n, \underline c}(\mu)), N^{\mathfrak p}(\phi_{\underline n, \underline c}(\lambda)))\cong \Hom_{\mathscr H_{\ell,r}}(\pi_{\underline n, \underline c}(L(\phi_{\underline n, \underline c}(\mu))), \pi_{\underline n, \underline c}(N^{\mathfrak p}(\phi_{\underline n, \underline c}(\lambda))))
\end{equation}
 for $ \mu \in \Lambda_{J, \underline n, \underline c}$
and $ \lambda\in \Lambda^o$.  When $r\le \min\{n_1, \ldots, n_\ell\}$, by  Corollary~\ref{picell}, \eqref{isomoandc} and Proposition~\ref{dualofcells},
 $$0\neq \pi_{\underline n, \underline c}(L(\tilde\lambda))\cong \text{Soc} (\tilde {S}^{\underline c}(\lambda'))\cong \text{hd}(S^{\underline c}(\lambda)), $$  forcing  $\pi_{\underline n, \underline c}(L(\tilde\lambda))\cong D^{\underline c}(\lambda)$.\end{proof}

\section{Isomorphisms between simple $\H_{\ell, r}$-modules}

The aim of this section is to give explicit isomorphisms between simple $\mathscr H_{\ell, r}$-modules defined via various cellular bases.
We start by recalling some results from
\cite[section~3]{BK}.
For each $\alpha\in Q_+$, let $\mathscr H_{\ell, \alpha }$ be defined  in \eqref{block1}.
Given $i\in J$, there is an algebra homomorphism
$$ \tau_{\alpha,\alpha_i}: \mathscr H_{\ell, \alpha }\rightarrow \mathscr H_{\ell, \alpha+\alpha_i}$$
which maps  $e_\alpha$
to $e_{\alpha,\alpha_i}:= \sum_{\mathbf i\in J^{\alpha+\alpha_i}, i_{r+1}=i}e(\mathbf i)$. Associated to $\tau_{\alpha,\alpha_i} $, there are
 exact induction and restriction functors
\begin{equation}\label{eifunctor}
e_i : \mathscr H_{\ell, \alpha+\alpha_i}\text{-mod}\rightarrow  \mathscr H_{\ell, \alpha}\text{-mod}
\end{equation}
and
\begin{equation}\label{fifunctor}
f_i : \mathscr H_{\ell, \alpha}\text{-mod}\rightarrow  \mathscr H_{\ell, \alpha+\alpha_i}\text{-mod},
\end{equation}
where $e_iM=e_{\alpha,\alpha_i}M$ is  viewed as an $\mathscr H_{\ell,\alpha}$-module via $ \tau_{\alpha,\alpha_i} $, and $f_iN=\mathscr H_{\ell, \alpha+\alpha_i}e_{\alpha,\alpha_i}\otimes _{\mathscr H_{\ell,\alpha}}N$, for
any $ M\in \mathscr H_{\ell, \alpha+\alpha_i}\text{-mod}$ and $ N\in\mathscr H_{\ell, \alpha}\text{-mod}$.
By \cite[section~3]{BK}, $\sum_{\alpha\in Q_+}[\mathscr H_{\ell,\alpha}\text{-mod}] $
is a left  $\mathfrak {sl}_J$-module  such that  $e_{i,i+1}$ and $e_{i+1,i}$  act on $\sum_{\alpha\in Q_+}[\mathscr H_{\ell,\alpha}\text{-mod}] $ via
 the functors $e_i$ and $f_i$ in \eqref{eifunctor}--\eqref{fifunctor}, respectively.
Let $\Lambda_{\ell,r}^{{\underline{n}}}(\alpha)=\{ \lambda\in \Lambda_{\ell,r}^{{\underline{n}}}\mid  e_\alpha \tilde {S}^{\underline c}(\lambda')\neq 0\}$.
Write
$$\bar{\Lambda}_{\ell,r}^{{\underline{n}}}:= \bigcup_{\alpha\in Q_+, \text{ht}(\alpha)=r} \Lambda_{\ell,r}^{{\underline{n}}}(\alpha).$$
  We consider quaternary pair $(\Lambda_{J, \underline n, \underline c},  \mathcal C_{J, \underline n, \underline c}, \Lambda^{\underline n, \underline c}V_{J}, \bar\Lambda_{\ell,r}^{{\underline{n}}})$ and
 $(\Lambda_{J, \tilde{\underline n}, \underline{c_0}}, \mathcal C_{J, \tilde{\underline n}, \underline{c_0}}, \Lambda^{\tilde{\underline n}, \underline{c_0}}V_{J}, \bar\Lambda_{\ell,r}^{\tilde{\underline n}})$, where
 $\tilde n$ is given in \eqref{definitionoftildeni}.
 By \eqref{philambda},  $$\phi_{\underline n, \underline c}(\lambda)=\sum_{j=1}^{m+n}a^{\lambda,\underline n, \underline c}_j\delta_j-\rho, \ \ \text{ for all $\lambda\in \Lambda_{J, \underline n, \underline c}$,}$$ where
\begin{equation}\label{alambda}
(a_{p_{k-1}+1}^{\lambda,\underline n, \underline c},a_{p_{k-1}+2}^{\lambda,\underline n, \underline c},\ldots, a_{p_{k}}^{\lambda,\underline n, \underline c})=\begin{cases}
(i_{n_k}, i_{n_k-1},\ldots, i_1)\quad  &\text{ if } c_k=0,\\
(-i_1,-i_2,\ldots,-i_{n_k})\quad & \text{ if } c_k=1,
\end{cases}
\end{equation}
and $(i_1,i_2,\ldots, i_{n_k})$ is given in \eqref{vlambdatensor} for all $1\le k\le \ell$.
Let $\gamma_{\underline n, \underline c}: \Lambda_{J, \underline n, \underline c}\rightarrow\cup_{r=0}^{\infty} \bar\Lambda_{\ell,r}^{{\underline{n}}}$
be the bijective map such that $\gamma_{\underline n, \underline c}(\lambda)= (\lambda^{(1)}, \lambda^{(2)}, \ldots, \lambda^{(\ell)})$, where
\begin{equation}\label{lambdapa}
\lambda^{(k)}=(a_{p_{k-1}+1}^{\lambda,\underline n, \underline c},a_{p_{k-1}+2}^{\lambda,\underline n, \underline c},\ldots, a_{p_{k}}^{\lambda,\underline n, \underline c} )-(a_{p_{k-1}+1}^{\kappa_J, \underline n, \underline c },a_{p_{k-1}+2}^{\kappa_J,\underline n, \underline c},\ldots, a_{p_{k}}^{\kappa_J,\underline n, \underline c} ).\end{equation}
Let $ \eta: \bar\Lambda_{\ell,r}^{\tilde{\underline{n}}} \rightarrow  \bar\Lambda_{\ell,r}^{{\underline{n}, }}$  be the map such that
\begin{equation}\label{defofsigma} \eta(\gamma_{\tilde{\underline n}, \underline{c_0}}(\lambda))=\gamma_{\underline n, \underline c}(\lambda),\quad \text{ for  all }
\lambda \in\Lambda_{J, \underline n, \underline c}.\end{equation}
If $\text{wt}(v_{\lambda})=\alpha$ with $\text{ht}(\alpha)=r$ and $r\leq\min\{n_1, \ldots, n_\ell\}$,
by Lemma~\ref{moduleisomofhe}, we have
$\pi_{\underline n, \underline c}(M^{\mathfrak p}(\phi_{\underline n, \underline c}(\lambda)))\cong \tilde {S}^{\underline c}(\mu)$, where $\mu $ is the dual partition  of $\gamma_{\underline n, \underline c}(\lambda)$ such that $\widetilde{ \mu'}=\phi_{\underline n, \underline c}(\lambda)$ (see Definition~\ref{defofhigest}(a)).

\begin{lemma}\label{bijectionofp} If $\lambda=(\lambda^{(1)}, \ldots, \lambda^{(\ell)})\in \bar\Lambda_{\ell,r}^{\tilde{\underline{n}}}$, then
$\eta(\lambda)=(\mu^{(1)},\ldots, \mu^{(\ell)})$ where  $\mu^{(i)}=\lambda^{(i)}$ if $c_i=0$ and $\mu^{(i)}=(\lambda^{(i)})'$ if $c_i=1$.
\end{lemma}
\begin{proof}

Let $\bar \eta$ be the map in (\ref{defofsigma}) for the case  $\ell=1$. By (\ref{alambda})--(\ref{lambdapa})
$$\eta(\gamma_{\tilde{\underline n}, \underline{c_0}}(\lambda))=(\bar \eta(\gamma_{\tilde{ n}_1, 0}(\lambda^{(1)}) ), \ldots, \bar \eta(\gamma_{\tilde{ n}_\ell, 0}(\lambda^{(\ell)}) )).  $$
So,  it is enough to consider $\ell=1$. Moreover, by (\ref{alambda})--(\ref{lambdapa}), the partition $\lambda^{(k)}$ is independent of $\text{Inf} J $  In fact, if $\text{Inf} J$ is changed, then  all numbers in \eqref{alambda} are   shifted with a common number. By \eqref{lambdapa},
 every $\lambda^{(k)}$ is fixed. 
  For simplicity, we  assume $J=\{1,2,\ldots, N-1\}$ and hence $\text{Inf} J=1$. If $c=0$, by (3.7), (3.11) and (3.35)-(3.36), we have $\mu=\lambda$.
If $c=1$,  then $\tilde{\underline{n}}=N-n$.
If $v_{\lambda}=v_{i_1}\wedge v_{i_2}\wedge \ldots \wedge v_{i_{N-n}}$ and $\lambda\in\Lambda_{J, \tilde{ n}, 0}$,  by \eqref{conditionofj} and~\eqref{alambda}--\eqref{lambdapa}, we have
\begin{enumerate} \item $
\gamma_{\tilde{n},0}(\lambda)=(i_{N-n}-N+n, i_{N-n-1}-N+n+1, \ldots, i_1-1)$,
\item  $\gamma_{n,1}(\lambda)=(N-n+1-j_1, N-n+2-j_{2}, \ldots, N-j_n)$,
 \item $ j_1<j_2<\ldots< j_{n}$  and  $\{j_1,j_2,\ldots,j_{n}\}=J_+\setminus \{i_1,i_2,\ldots, i_{N-n}\}$.
 \end{enumerate}

If $|\gamma_{\tilde n,0}(\lambda)|=0$, then $\gamma_{\tilde n,0}(\lambda)=\emptyset$ and $(i_1, i_2, \ldots, i_{N-n})=(1, 2,\ldots, N-n)$. By (c),
$(j_1,j_2,\ldots,j_{n})=(N-n+1,N-n+2,\ldots,N)$ and hence  $\gamma_{n,1}(\lambda)=\emptyset$.
Suppose  the result holds for $ \gamma_{\tilde n ,0}(\lambda)$. We need to prove the result for any $\gamma_{\tilde n,0}(\mu)$, where  $\gamma_{\tilde n,0}(\mu)$ is obtained from $ \gamma_{\tilde n,0}(\lambda)$ by adding an \text{addable node}, say in the $l$th row. It is  in the $(\gamma_{\tilde n,0}(\lambda)_l+1)$th column of $\gamma_{\tilde n,0}(\mu)$.
Using  (a) yields
\begin{equation}\label{sigmalambdal} \gamma_{\tilde n,0}(\lambda)_l=i_{N-n-l+1}-(N-n-l+1).
\end{equation}
By induction assumption, $\gamma_{n,1}(\lambda)$ is the dual partition  of $
\gamma_{\tilde{n},0}(\lambda)$.  So, the $(\gamma_{\tilde n,0}(\lambda)_l+1)$th component  of $\gamma_{n,1}(\lambda)$ is $l-1$. By (b),
\begin{equation}\label{sigmalambdal2}
N-n+\gamma_{\tilde n,0}(\lambda)_l+1-j_{ \gamma_{\tilde n,0}(\lambda)_l+1} =l-1.
\end{equation}
By \eqref{sigmalambdal}--\eqref{sigmalambdal2}, we have
\begin{equation}\label{jandi}
j_{ \gamma_{\tilde n,0}(\lambda)_l+1}=i_{N-n-l+1}+1.
\end{equation}
Note that $\gamma_{\tilde n,0}(\mu)$ is obtained from  $\gamma_{\tilde n,0}(\lambda)$ by replacing $i_{N-n-l+1}$ with $i_{N-n-l+1}+1$. It follows from (a)-(c) and \eqref{jandi} that
$\gamma_{n,1}(\mu)$ is obtained from  $\gamma_{n,1}(\lambda)$ by replacing  $j_{ \gamma_{\tilde n,0}(\lambda)_l+1}$ with $j_{ \gamma_{\tilde n,0}(\lambda)_l+1}-1$. Now the result follows from \eqref{sigmalambdal2}
and induction assumption on $ \gamma_{\tilde n,0}(\lambda)$.
\end{proof}

\begin{example}\label{example111} Assume $(\ell, \underline{n}, \underline c, J_+)=(1, n, 1,\{1, 2, \ldots,  N\})$.
If  $$v_{\lambda}=v_{1}\wedge v_{2}\wedge \ldots \wedge v_{N-n-1}\wedge v_N\in \Lambda^{\tilde n} V_J,$$ then the corresponding $v_\lambda\in \Lambda^{n}W_J$ is
$w_{N-n}\wedge w_{N-n+1}\wedge \ldots \wedge w_{N-1}$ (see \eqref{conditionofj}).
By \eqref{alambda}--\eqref{lambdapa},
$\gamma_{\tilde n,0}(\lambda)=(n,0,0,\ldots,0)$ and $\gamma_{n,1}(\lambda)=(1^n)$ which is $\gamma_{\tilde n,0}(\lambda)'$.
\end{example}

The following result  is the first case of our main result. We explain  the idea   as follows. Using  Lemma~\ref{moduleisomofhe} and  Corollary~\ref{picell}, we establish an explicit relationship between parabolic
 (dual) Verma supermodules in $\mathcal O^{\underline n, \underline c}$  and the cell modules $ \tilde S^{\underline c}(\lambda')$'s of $\mathscr H_{\ell, r}$. This leads to Lemma~\ref{irrunderpi}, which gives an explicit relationship between  simple  modules in $\mathcal O^{\underline n, \underline c}$ and the simple $\mathscr H_{\ell, r}$-modules $D^{\underline c}(\lambda)$'s. For any $(\underline n, \underline c)$, we have a categorification  of the unique irreducible summand of $\wedge^{\underline n, \underline c} V_J$  with   highest weight $|\kappa_J|$ via  $\sum_{\alpha\in Q_+} [\mathscr H_{\ell, \alpha}\text{-mod}]$, where $\kappa_J$ is defined in \eqref{kaj}. Finally,  using Brundan-Losev-Webster's results on uniqueness of tensor product categorification (cf. Theorem~\ref{unicate}) and  the  $\mathfrak{sl}_J$-isomorphism $\wedge^{\underline n, \underline c} V_J \cong \wedge^{\underline {\tilde n}, \underline {c_0}} V_J$, we determine whether  $D^{\underline c} (\lambda)$  is isomorphic to  $D^{\underline {c_0}}(\mu)$ or not.

 Recall that $\Lambda^o$ and  $\gamma_{\tilde{\underline n}, \underline{c_0}}$ are given in Lemma~\ref{irrunderpi} and  \eqref{lambdapa} respectively.
\begin{Theorem}\label{main1} Suppose $\lambda\in \Lambda_{\ell, r}$. Then  $D^{\underline {c_0}}(\lambda)\neq 0$ if and only if  $\lambda\in \gamma_{\tilde{\underline n}, \underline {c_0}}(\Lambda^o)$ for some  $\underline n$ such that $r\le \min \{n_1, \ldots, n_\ell\}$. Moreover,  $D^{\underline {c_0}}(\lambda)\cong D^{\underline c}(\eta(\lambda))$ for any  $\lambda\in \gamma_{\tilde{\underline n}, \underline {c_0}}(\Lambda^o)$, where  $\eta(\lambda)$ is given in Lemma~\ref{bijectionofp}.
\end{Theorem}
\begin{proof}We consider $\underline n$ such that $|\lambda|\le \min\{n_1, n_2, \ldots, n_\ell\}$. In this case,  the first result follows from Lemma~\ref{irrunderpi}.  Note that  $\Lambda_{J, \underline n, \underline c}$ can be identified  with $\Lambda_{J, \tilde{\underline n}, \underline {c_0}}$.
By  Theorem~\ref{unicate},
there is an  isomorphism $\varphi: [\mathcal C_{J,\underline n, \underline c}]\cong [\mathcal C_{J,\tilde{\underline n}, \underline {c_0}}]$ such that
  for any $\mu \in\Lambda_{J, \underline n, \underline c}$,
\begin{equation}\label{key113}\varphi([N^{\mathfrak p}(\phi_{\underline n, \underline c }(\mu))])=[ N^{\mathfrak p}(\phi_{\tilde{\underline n}, \underline {c_0} }(\mu))], \text{ and }  \varphi([L(\phi_{\underline n, \underline c }(\mu))])=[L(\phi_{\tilde{\underline n}, \underline {c_0} }(\mu))].\end{equation}
Define
\begin{equation}\label{defofschurpi}
\pi_{\underline n, \underline c }=\Hom_{\mathcal C_{J,\underline n, \underline c}}(
 \bigoplus_{r\in \mathbb N} T^{J,  r},-), \text{  and   $\pi_{ \tilde{\underline n}, \underline {c_0}} =\Hom_{\mathcal C_{J,\tilde{\underline n}, \underline {c_0}}}(\bigoplus_{r\in \mathbb N}T^{J,  r},-)$.}\end{equation}
 Then  both $\pi_{\underline n, \underline c }$ and $\pi_{ \tilde{\underline n}, \underline {c_0}} $ are exact and induce linear maps from $[C_{J,\underline n, \underline c}]$ and
 $[C_{J,\tilde{\underline n}, \underline {c_0}}]$  to $\sum_{\alpha\in Q_+}[\mathscr H_{\ell,\alpha}\text{-mod}] $, respectively,  where
 $\mathscr H_{\ell,\alpha}$ is the block of $\mathscr H_{\ell,r}$ with respect to $\alpha$ (see \eqref{block1}).
   Moreover, by  Lemma~\ref{polyofx} and  $\mathscr H_{\ell,r-1}$-isomorphism
  $$\Hom_{\mathcal C_{J,\underline n, \underline c}}(
  T^{J,  r},N^{\mathfrak p}(\phi_{\underline n, \underline c }(\mu))) \cong  \Hom_{\mathcal C_{J,\underline n, \underline c}}(
  T^{J,  r-1},N^{\mathfrak p}(\phi_{\underline n, \underline c }(\mu))\otimes U^*),$$
   $\pi_{\underline n, \underline c }$ and $\pi_{ \tilde{\underline n}, \underline {c_0}} $ are both $\mathfrak {sl}_J$-epimorphisms and the kernels of  $\pi_{\underline n, \underline c }$ and $\pi_{ \tilde{\underline n}, \underline {c_0}} $ are spanned by  $\{[L(\phi_{\underline n, \underline c }(\mu))]\mid \mu \in\Lambda_{J, \underline n, \underline c} \text{ and }  \mu\notin \Lambda^o\}$ and $\{[L(\phi_{\tilde{\underline n}, \underline {c_0} }(\mu))]\mid \mu \in\Lambda_{J, \tilde{\underline n}, \underline {c_0}} \text{ and }  \mu\notin \Lambda^o\}$, respectively (see Lemma~\ref{irrunderpi}).  Let $L_{\underline n, \underline c}:=[\mathcal C_{J, \underline n, \underline c}]/\text{ker }\pi_{\underline n, \underline c } $. By Theorem~\ref{unicate} and \cite[Corollary~2.3]{BK},  both $L_{\underline n, \underline c}$ and $L_{\tilde{\underline n}, \underline {c_0}}$ are isomorphic to $V(|\kappa_J|)$, the simple  $\mathfrak {sl}_J$-module with highest weight $|\kappa_J|$. Moreover, the corresponding isomorphisms send $ [L(\phi_{\underline n, \underline c }(\mu))]$ and $ [L(\phi_{\tilde{\underline n}, \underline {c_0} }(\mu))]$ to $b_\mu^*$ for any $\mu\in \Lambda^o$,
 where $\{b_\mu^*\mid \mu\in \Lambda^o \}$ is the dual canonical basis of $V(|\kappa_J|)$, which is the image of the canonical map from $\Lambda^{\tilde{\underline n}, \underline {c_0}  } V_J $ to $V(|\kappa_J|)$.

Abusing  notation, let  $\pi_{\underline n, \underline c}$ be the  isomorphism from $L_{\underline n, \underline c} $ to $\sum_{\alpha\in Q_+}[\mathscr H_{\ell, \alpha}\text{-mod}] $.  Then  there is an $\mathfrak {sl}_J$-automorphism of the irreducible module with highest weight $|\kappa_J|$
 $$ \pi_{\underline n, \underline c}\circ \varphi\circ \pi^{-1}_{\tilde{\underline n}, \underline {c_0}}:\sum_{\alpha\in Q_+}[\mathscr H_{\ell, \alpha}\text{-mod}] \rightarrow\sum_{\alpha\in Q_+}[\mathscr H_{\ell, \alpha}\text{-mod}]. $$
    Suppose $\lambda \in\Lambda_{J, \underline n, \underline c}$.  If $\text{wt}(v_{\lambda})=\alpha$ with $\text{ht}(\alpha)=r$ and $r\leq\min\{n_1, \ldots, n_\ell\}$, by Corollary~\ref{picell}, \eqref{isomoandc} and  Lemma~\ref{irrunderpi}, we have
 \begin{enumerate} \item  $\pi_{\underline n, \underline c }([N^{\mathfrak p}(\phi_{\underline n, \underline c }(\lambda))])=[\tilde {S}^{\underline c}(\mu)]$ and $\pi_{\underline n, \underline c }([L(\phi_{\underline n, \underline c }(\lambda))])=[D^{\underline c}(\mu')]$, \item
  $\pi_{ \tilde{\underline n}, \underline {c_0}}([ N^{\mathfrak p}(\phi_{\tilde{\underline n}, \underline {c_0} }(\lambda))=[\tilde {S}^{\underline {c_0}}(\nu)]$ and $\pi_{ \tilde{\underline n}, \underline {c_0}}([L(\phi_{\tilde{\underline n}, \underline {c_0} }(\lambda))])=[D^{\underline {c_0}}(\nu')]$, \end{enumerate}
\noindent where $\mu$ (resp., $\nu$) is the dual partition of $\gamma_{\underline n, \underline c}(\lambda)$
 (resp., $\gamma_{\underline{ \tilde n}, \underline {c_0}}(\lambda))$.  By (a)-(b) and \eqref{key113}, $\pi_{\underline n, \underline c}\circ \varphi\circ \pi^{-1}_{\tilde{\underline n}, \underline {c_0}}$  sends   $[D^{\underline {c_0}}(\gamma_{\tilde{\underline n}, \underline {c_0}}(\lambda))]$  to $[D^{\underline c}(\gamma_{\underline n, \underline c}(\lambda))]$.
 Since  $\gamma_{\underline n,\underline c}(\kappa_J)=\gamma_{\tilde{\underline n}, \underline {c_0}}(\kappa_J)=\emptyset$ and  $[D^{\underline c}(\emptyset)]$ is the unique highest weight vector of the irreducible module with highest weight $|\kappa_J|$, we have that
 $ \pi_{\underline n, \underline c}\circ \varphi\circ \pi^{-1}_{\tilde{\underline n}, \underline {c_0}}$ is the identity  map and hence  $[D^{\underline {c_0}}(\gamma_{\tilde{\underline n}, \underline {c_0}}(\lambda))]= [D^{\underline c}(\gamma_{\underline n, \underline c}(\lambda))]$,  forcing  $D^{\underline {c_0}}(\gamma_{\tilde{\underline n}, \underline {c_0}}(\lambda))\cong  D^{\underline c}(\gamma_{\underline n, \underline c}(\lambda))$.
 Now the result follows from Lemma~\ref{bijectionofp} and \eqref{defofsigma}.
\end{proof}

Finally, we consider the case when the  $\o$ is replaced by the $\o^\xi$ in \eqref{omxi} for   any  $\xi\in\mathfrak S_\ell$, where $\o=(\omega_1, \ldots, \omega_\ell)$.  For this purpose, we assume
$\omega_i\in \mathbb Z$ such that  $\omega_1\geq\ldots\geq \omega_\ell$.
Suppose $\lambda\in \Lambda_{\ell, r}$. Recall that  $S^\xi(\lambda)$
(resp., $\tilde {S}^\xi(\lambda)$)
is  the associated  right cell module of
$\H_{\ell, r}$ with respect to the cellular basis of $\H_{\ell, r}$ in  Corollary~\ref{Hc cellular}(c) (resp., (d)).
 The corresponding simple head is denoted by $D^\xi({\lambda})$ (resp., $\tilde D^\xi({\lambda}) $).
Let $$\Lambda=\Lambda_{\omega_1}+\Lambda_{\omega_2}+\ldots+\Lambda_{\omega_\ell},$$
where $\Lambda_i$ is the fundamental dominant weight of $\mathfrak {sl}_I$
for  a bounded below   interval $I$ of $\mathbb Z$.
Let Inf $I$ be the minimal element in $I$.
Then the simple $\mathfrak {sl}_I$-module  $V(\Lambda_{\omega_i})$ with highest weight $\Lambda_{\omega_i}$ is isomorphic to  $ \bigwedge^{\omega_i-\text{Inf }I+1}V_I$. Let $n_i=\omega_i-\text{Inf }I+1$. Then $n_1\ge \ldots \ge n_\ell$.  Let \begin{equation}\label{wedgemodule}
F(\Lambda)^{\xi}:=V(\Lambda_{\omega_{(1)\xi}})\otimes V(\Lambda_{\omega_{(2)\xi}})\otimes \ldots \otimes  V(\Lambda_{\omega_{(\ell)\xi}})\cong \bigwedge^{n_{(1)\xi}}V_I\otimes \bigwedge^{n_{(2)\xi}}V_I\otimes \ldots\otimes \bigwedge^{n_{(\ell)\xi}}V_I.
\end{equation}
Then the simple $\mathfrak{sl}_I$-module $V(\Lambda)$ with highest weight $\Lambda$  is a direct summand of $F(\Lambda)^{\xi}$,  for all $ \xi\in\mathfrak S_\ell$. Let $\pi^\xi: F(\Lambda)^{\xi}\rightarrow V(\Lambda)$ be  the  canonical projection.

 Following \cite[Section~2]{Br}, a
 $(\Lambda,\xi)$-tableau is the down-justified tableau with $n_{(i)\xi}$ boxes in the $i$-th column and each entry is an integer. Let $A(i,j)$ be  the entry in the $i$th row and $j$th column of $A$. If $A(i, j)\in I_+$ such that $A(i, j)>A(i+1, j)$ for    $1\le i\le n_{(j)\xi}$ and  $1\le j\le \ell$, then  $A$ is called a  \textit{column-strict tableau}. Let $\text{Col}^{\Lambda,\xi}$ be the set of all column-strict $(\Lambda,\xi)$-tableaux. For any $A\in\text{Col}^{\Lambda,\xi}$, let $$\gamma(A)=(a_1,a_2,\ldots, a_n)$$ be obtained by reading the entries of $A$ from top to bottom along the columns. Then  $n=\sum_{i=1}^\ell n_i$.
Let $P(\gamma(A))$ be the tableau corresponds to the first part of the  image of the word $\gamma(A)$ under the  Robinson-Schensted-Knuth  correspondence (see, e.g. \cite[Section~4.1]{F}). This is a  tableau
$$\emptyset\leftarrow a_1\leftarrow a_2\leftarrow\ldots \leftarrow a_n $$ of some partition, where $ ``\leftarrow"$ denotes the row insertion as in \cite[Section~1.1]{F}. See  Example~\ref{exampleofp}.
A column-strict tableau  $A$ is called standard if $P(\gamma(A))$  is of type $\underline n'$, where $\underline n'$ is the dual partition of $\underline n$. This makes sense since we are assuming  $n_1\geq n_2\geq \ldots\geq n_\ell$.

Let  $\text{Std}^{\Lambda,\xi}$ be  the set of all standard tableaux.
For any $A\in\text{Col}^{\Lambda,\xi}$,
define $$M^\xi_A:= v_{A(1,1)}\wedge v_{A(2,1)}\wedge \ldots \wedge v_{A(n_{(1)\xi},1)}\otimes \ldots \otimes v_{A(1,\ell)}\wedge v_{A(2,1)}\wedge \ldots\wedge v_{A(n_{(\ell)\xi},\ell)}.$$
Then $M_A^\xi$  corresponds to $v_\lambda$ for
$\lambda\in \Lambda_{I,\underline n,\underline {c_0}}$ such that
$\lambda_{ij}=1$ for $j=A(i,k)$, $k=1,2,\ldots,
n_{(i)\xi}$,  and $\lambda_{ij}=0$ otherwise.
 Then $\{M^\xi_A\mid A\in  \text{Col}^{\Lambda,\xi}\}$ is known as the monomial  basis of   $F(\Lambda)^{\xi}$. Let $\{L^\xi_A\mid A\in  \text{Col}^{\Lambda,\xi}\}$ be the   dual canonical basis of $F(\Lambda)^{\xi}$.
For any $A\in  \text{Col}^{\Lambda,\xi}$,  let $ S_A^\xi:=\pi^\xi( M^\xi_A)$ and $D^\xi_A:=\pi^\xi(L^\xi_A)$.
\begin{lemma}\label{isomofcrystal}\cite[Section~2, Theorem~26]{Br}
\begin{enumerate}
 \item There is a crystal graph structure on $\text{Col}^{\Lambda,\xi}$, which  corresponds to $F(\Lambda)^{\xi}$ such that $\text{Std}^{\Lambda,\xi}$ is a sub-crystal which corresponds  to $V(\Lambda)$.
Moreover, there is an isomorphism of crystals $R:\text{Std}^{\Lambda,\xi}\rightarrow  \text{Std}^{\Lambda,1}$ such that
the entries from top to bottom of the $i$th column of  $R(A)$
 are the entries from bottom to top  of the $i$th column  of $P(\gamma(A))$.
\item $ \pi^\xi(L^\xi_A)\neq 0$ if and only if $A\in \text{Std}^{\Lambda,\xi}$.
\item $\{S_A^\xi\mid A\in \text{Std}^{\Lambda,\xi}\}$ is  the monomial    basis of $V(\Lambda)$.
 \item  $\{D_A^\xi\mid A\in \text{Std}^{\Lambda,\xi}\}$ is the  dual canonical  basis of $V(\Lambda)$.
\item $D^\xi_A=D^1_{R(A)}$.
\end{enumerate}
\end{lemma}

For any $\xi\in \mathfrak S_\ell$, let $\mathfrak p^\xi$ be the parabolic subalgebra of $\mathfrak {gl}_n$ with respect to the Levi subalgebra
 $\mathfrak{gl}_{n_{(1)\xi}}\oplus \mathfrak{gl}_{n_{(2)\xi}}\oplus  \ldots\oplus  \mathfrak{gl}_{n_{(\ell)\xi}}$.
Let $\mathcal O^{\mathfrak p^\xi}$ be the corresponding parabolic category $\mathcal O$.
For any $A\in \text{Col}^{\Lambda,\xi} $, if $\gamma(A)=(a_1,a_2,\ldots, a_n)$, then $\sum_{i=1}^n  (a_{i}+i-1 )\delta_i$ is a $\mathfrak p^\xi$-dominant weight.  The corresponding parabolic Verma, simple and  dual Verma modules with  highest weight $\sum _{i=1}^n (a_i+i-1)\delta_i$
are denoted by
$M^{\xi}(A)$, $L^\xi(A)$ and $N^\xi{(A)}$, respectively. For any $\alpha\in Q_+$, let $$\text{Col}_\alpha^{\Lambda,\xi}=\{A\in \text{Col}^{\Lambda,\xi}\mid \text{wt}(M^\xi_A)=\Lambda-\alpha\}, $$
where $\text{wt}(M^\xi_A)$ is the weight of $M^\xi_A$.
Let $\mathcal O_{\alpha}^{\Lambda,\xi}$ be the Serre subcategory of $\mathcal O^{\mathfrak p^\xi}$ with simple objects $\{L^\xi(A)\mid A\in \text{Col}_\alpha^{\Lambda,\xi}\}$.
When $\text{Col}_\alpha^{\Lambda,\xi}$ is not empty, $\mathcal O_{\alpha}^{\Lambda,\xi}$ is a single block of $\mathcal O^{\mathfrak p^\xi}$(see \cite[Theorem~2]{BR}).
Let $\mathcal O^{\Lambda,\xi}:= \bigoplus_{\alpha\in\ Q_+}\mathcal O_{\alpha}^{\Lambda,\xi}.$

\begin{Theorem}\label{bkcate}\cite[Theorem~3.1]{BK} There is an $\mathfrak {sl}_I$-isomorphism $ \varphi:F(\Lambda)^{\xi}\rightarrow [\mathcal O^{\Lambda,\xi} ]$
such that $\varphi(M^\xi_A)=[M^{\xi}(A)]$ and $\varphi(L^\xi_A)=[L^{\xi}(A)]$.

\end{Theorem}

Let $A^\xi$ be the unique element in $\text{Col}_0^{\Lambda,\xi}$. Then $L^\xi(A^\xi)= M^\xi(A^\xi)$. It is  projective, injective and hence tilting.
Let $T^{\Lambda,\xi,r}$ be  obtained from $T^{J, r}$ in \eqref{tens321} by using $L^\xi(A^\xi)$ instead of $L(\lambda_J)$.
Recall the Schur functor $\pi_{(\underline n)\xi, \underline {c_0} }$ in \eqref{defofschurpi}.

\begin{lemma}\label{mapofpi}
For any $A\in \text{Col}_\alpha^{\Lambda,\xi}$, let $\lambda_A=(\lambda^{(1)},\lambda^{(2)},\ldots,\lambda^{(\ell)})$, where
 $$\lambda^{(i)}=(A(1,i)-A^\xi(1,i), A(2,i)-A^\xi(2,i),\ldots,A(n_{(i)\xi},i)-A^\xi(n_{(i)\xi},i))$$
for $1\leq i\leq \ell$. Then $\lambda_A$ is an  $\ell$-partition. Moreover, if $\lambda_A\in\Lambda_{\ell,r}$ with $r\leq\min\{n_1, \ldots, n_\ell\}$, then

 \begin{enumerate} \item
$\pi_{(\underline n)\xi, \underline {c_0} }(N^\xi(A))=\tilde S^\xi((\lambda_A)')$,   \item $\pi_{(\underline n)\xi, \underline {c_0} }(L^\xi(A))$ is $D^\xi(\lambda_A)$ if $A\in\text{Std}^{\Lambda,\xi}$, and $0$, otherwise.
\end{enumerate}\end{lemma}
\begin{proof} It is easy to see that $\lambda_A$ is an $\ell$-partition. Now (a) follows from  Corollary~\ref{picell} and  (b) follows from Lemma~\ref{irrunderpi}. \end{proof}
\begin{example}\label{exampleofp}
Assume $\underline n=(3,2,2)$ , $I=\{1,2,\ldots, N-1\}$ such that  $N\gg0$. Then   $\o=(3,2,2)$, $F(\Lambda)=\Lambda^3V_I\otimes \Lambda^2V_I\otimes \Lambda^2V_I$ and $\underline n'=(3,3,1)$.
 Let $\xi=s_1\in \mathfrak S_3$. Then    $\o^{s_1}=(2,3,2)$, $F(\Lambda)^{s_1}=\Lambda^2V_I\otimes \Lambda^3V_I\otimes \Lambda^2V_I$ and

\begin{picture}(10, 5.6)%
\put(30,-2){\makebox(0,0){ $A^{s_1}$=}}
\put(35,-4){\line(1,0){12}}
\put(35,0){\line(1,0){12}}
\put(39,4){\line(1,0){4}}
\put(35,-8){\line(1,0){12}}
\put(35,0){\line(0,-1){8}}
\put(39,4){\line(0,-1){12}}
\put(43,4){\line(0,-1){12}}
\put(47,0){\line(0,-1){8}}
\put(52.8,-2){\makebox(0,0){and}}

\put(36.8,-1.8){\makebox(0,0){2}}\put(36.8,-5.8){\makebox(0,0){1}}
\put(40.8,-1.8){\makebox(0,0){2}}\put(40.8,-5.8){\makebox(0,0){1}}
\put(44.8,-1.8){\makebox(0,0){2}}\put(44.8,-5.8){\makebox(0,0){1}}
\put(40.8,1.8){\makebox(0,0){3}}
\put(69.8,-2){\makebox(0,0){$ ~A^1=\young(3,222,111)$ ~.}}

\end{picture}\vspace{8mm}

\noindent Suppose  $\alpha = 3\alpha_2+\alpha_3$ and $\beta= \alpha_1+2\alpha_2+\alpha_3$. Pick

\begin{picture}(10, 5.6)
\put(28,-2){\makebox(0,0){  $A$=}}
\put(35,-4){\line(1,0){12}}
\put(35,0){\line(1,0){12}}
\put(39,4){\line(1,0){4}}
\put(35,-8){\line(1,0){12}}
\put(35,0){\line(0,-1){8}}
\put(39,4){\line(0,-1){12}}
\put(43,4){\line(0,-1){12}}
\put(47,0){\line(0,-1){8}}
\put(36.8,-1.8){\makebox(0,0){3}}\put(36.8,-5.8){\makebox(0,0){1}}
\put(40.8,-1.8){\makebox(0,0){3}}\put(40.8,-5.8){\makebox(0,0){1}}
\put(44.8,-1.8){\makebox(0,0){3}}\put(44.8,-5.8){\makebox(0,0){1}}
\put(40.8,1.8){\makebox(0,0){4}}
\put(59.8,-2){\makebox(0,0){$\in \text{Col}_\alpha^{\Lambda,s_1}$ and }}

\put(75,-2){\makebox(0,0){$B$=}}
\put(80,-4){\line(1,0){12}}
\put(80,0){\line(1,0){12}}
\put(84,4){\line(1,0){4}}
\put(80,-8){\line(1,0){12}}
\put(80,0){\line(0,-1){8}}
\put(84,4){\line(0,-1){12}}
\put(88,4){\line(0,-1){12}}
\put(92,0){\line(0,-1){8}}
\put(81.8,-1.8){\makebox(0,0){3}}\put(81.8,-5.8){\makebox(0,0){1}}
\put(85.8,-1.8){\makebox(0,0){3}}\put(85.8,-5.8){\makebox(0,0){2}}
\put(89.8,-1.8){\makebox(0,0){2}}\put(89.8,-5.8){\makebox(0,0){1}}
\put(85.8,1.8){\makebox(0,0){4}}
\put(100.8,-2){\makebox(0,0){$\in \text{Col}_\beta^{\Lambda,s_1}.$}}
\end{picture}\vspace{8mm}

\noindent
Then $\gamma(A)=(3,1,4,3,1,3,1)$, $M^{s_1}_A=v_3\wedge v_1\otimes v_4\wedge v_3\wedge v_1\otimes v_3\wedge v_1$, with weight $\text{wt}(M_A^{s_1})=3\delta_1+3\delta_3+\delta_4$. Moreover,  $M^{s_1}_A$  corresponds to $v_\lambda$ in \eqref{defofvlambda1} with $\lambda=(\lambda^{(1)}, \lambda^{(2)}, \lambda^{(3)})$ such that  $\lambda^{(1)}=\lambda^{(3)}=(1,0,1,0,\ldots,0)$ and $\lambda^{(2)}=(1,0,1,1,0,\ldots,0)$. We compute $P(\gamma(A))$ as follows.
$$ \emptyset  \overset{3} \leftarrow  \young(3) \overset{1} \leftarrow \young(1,3)\overset {4} \leftarrow \young(14,3)\overset {3} \leftarrow \young(13,34)\overset{1} \leftarrow \young(11,33,4)
\overset {3} \leftarrow \young(113,33,4)\overset {1} \leftarrow \young(111,333,4)=P(\gamma(A)), $$
which is of type $(3,3,1)$.
So, $A\in \text{Std}^{\Lambda,s_1}$ and $R(A)= \young(4,333,111)$.
Let $\lambda_A=(\lambda_A^{(1)}, \lambda_A^{(2)},\lambda_A^{(3)})$. Then $\lambda_A^{(1)}=\lambda_A^{(3)}=(3,1)-(2,1)=(1,0)$, $\lambda_A^{(2)}=(4,3,1)-(3,2,1)=(1,1,0)$.
Similarly, we have $\lambda_{R(A)}=((1,1),(1),(1))$.
For $B$,  we have $$P(\gamma(B))=\young(112,23,3,4)$$ which is not of type $(3,3,1)$. So,  $B\notin \text{Std}^{\Lambda,s_1}$.
\end{example}

The following result follows from  Lemma~\ref{irrunderpi} and the arguments in the proof of Theorem~\ref{main1}, immediately.
\begin{Prop}\label{cateofhcxi} Suppose $\xi\in\mathfrak S_\ell$.
Let $\psi=\pi_{(\underline n)\xi, \underline {c_0} }\circ \varphi|_{V(\Lambda)}: V(\Lambda) \rightarrow \bigoplus_{\alpha\in Q_+}[\mathscr H_{\ell,\alpha}\text{-mod }]$,
 where $\varphi$ is given  in Theorem~\ref{bkcate}. Then  $\psi$ is an $\mathfrak{sl}_I$--isomorphism. Moreover,
  \begin{enumerate} \item $\psi([S^{\xi}_A])=[\pi_{(\underline n)\xi, \underline {c_0} }(N^\xi(A))]$ if  $A\in \text{Col}_\alpha^{\Lambda,\xi}$,
  \item  $\psi([D^{\xi}_A])=[\pi_{(\underline n)\xi, \underline {c_0} }(L^\xi(A))]$ if
 $A\in \text{Std}_\alpha^{\Lambda,\xi}$.
\end{enumerate}
\end{Prop}

The following result is the second case  of our main result. The idea of the proof is similar to that of Theorem~\ref{main1}. The difference   is that we use the $\mathfrak{sl}_I$-isomorphism $F(\Lambda)\cong F(\Lambda)^\xi$
 and the isomorphism of crystals of the irreducible $\mathfrak{sl}_I$-module $V (\Lambda) $ (which is also the unique irreducible direct summand of $F(\Lambda )^\xi$ with highest weight $\Lambda$) in Lemma~\ref{isomofcrystal}(e).

\begin{Theorem}\label{main2} $D^\xi(\lambda_A)\cong D^1(\lambda_{R(A)})$ for any $\xi\in\mathfrak S_\ell$ and $A\in\text{Std}^{\Lambda,\xi} $. \end{Theorem}
\begin{proof}

By Proposition~\ref{cateofhcxi}, there are two  $\mathfrak {sl}_I$-isomorphisms
$$
\varphi^\xi, \varphi^1 : V(\Lambda)\rightarrow\bigoplus_{\alpha\in Q_+}[\mathscr H_{\ell,\alpha}\text{-mod }]
$$
where $\varphi^\xi=\pi_{(\underline n)^\xi, \underline {c_0} }\circ \varphi|_{V(\Lambda)} $ and $\varphi^1=\pi_{\underline n, \underline {c_0} }\circ \varphi|_{V(\Lambda)}$.
This implies the existence of an  $\mathfrak {sl}_I$-automorphism
$$ \varphi^1\circ(\varphi^\xi)^{-1} : \bigoplus_{\alpha\in Q_+}[\mathscr H_{\ell,\alpha}\text{-mod }]\rightarrow \bigoplus_{\alpha\in Q_+}[\mathscr H_{\ell,\alpha}\text{-mod }].$$
Note that $\bigoplus_{\alpha\in Q_+}[\mathscr H_{\ell,\alpha}\text{-mod }]\cong V(\Lambda)$.
By Proposition~\ref{cateofhcxi}(b), $$ \varphi^1\circ(\varphi^\xi)^{-1}([D^\xi(\emptyset )] )=[D^\xi(\emptyset)]=[D^1(\emptyset)],$$
where $[D^\xi(\emptyset)]$ is the unique highest weight vector of $\bigoplus_{\alpha\in Q_+}[\mathscr H_{\ell,\alpha}\text{-mod }]$.
So, $\varphi^1\circ(\varphi^\xi)^{-1}=1$, the identity map. Assume that $ D^\xi(\lambda_A)\cong D^1(\mu)$ for some $\ell$-multipartition $\mu$.
Then $(\varphi^\xi)^{-1}([D^\xi(\lambda_A)])=D_{\lambda_A}^\xi $ and $(\varphi^1)^{-1}([D^1(\mu)])=D_{\mu}^1 $. Since $\varphi^1\circ(\varphi^\xi)^{-1}=1$, we have $D_{\lambda_A}^\xi=D_{\mu}^1$. By
Lemma~\ref{isomofcrystal}(e), $\mu =\lambda_{R(A)}$.
\end{proof}

\begin{rem} When $\xi$ is the longest element of $\mathfrak S_\ell$, we have   $\omega^{\xi}=(\omega_\ell,\omega_{\ell-1}, \ldots,\omega_1)$.  In this case,    Theorem~\ref{main2} is \cite[Theorem~4.15(ii)]{BK}, whose proof depends on  Arkhipov's twisting functor
 to relate $M^{\xi}(A)$ with  $N^1(A)$. In general,  we do not know  whether one can use   Arkhipov's twisting functor to relate  $M^{\xi}(A)$ and $N^1(A)$.
 Our point is to use explicit descriptions on highest weight vectors of  $L^\xi(A^\xi)\otimes V^{\otimes r}$ (see  Theorem~\ref{hiofcyche})  to establish explicit relationships between parabolic Verma modules and cell modules. Finally, by Theorem~\ref{main1} and Theorem~\ref{main2},
 we know explicit relationships between simple modules defined via various cellular bases in Corollary~\ref{Hc cellular}.  When $\underline c=1^\ell$ and  $\xi$
 is the longest element in $\Sym_\ell$, the results of Theorem~3.15 and Theorem~3.21 (for Ariki-Koike algebras) yield  the generalized Mullineux involution in \cite{JL} when  $q$ is not a root of unity.
In this  case,
the relationship between $D^{\underline {c_0}}(\lambda)$ and $D^{\underline c}(\mu)$ given in Theorem~3.15 is the
first step in \cite[4.4]{JL}, and the relationship between $D^{\underline {c_0}}(\lambda)$ and $D^{\xi}(\mu)$ given in Theorem~3.21 is the second step in \cite[4.4]{JL}.
\end{rem}

\end{document}